\documentclass[10pt,conference]{IEEEtran}
\usepackage{cite}
\usepackage{amsmath,amssymb,amsfonts,amsthm}
\usepackage{algorithmic}
\usepackage{graphicx}
\usepackage{textcomp}
\usepackage{xcolor}
\def\BibTeX{{\rm B\kern-.05em{\sc i\kern-.025em b}\kern-.08em
    T\kern-.1667em\lower.7ex\hbox{E}\kern-.125emX}}

\usepackage{hyperref}
\usepackage{quiver}

\usepackage{mathrsfs}
\usepackage{bm}
\usepackage{adjustbox}

\usepackage{tikzit}				

\tikzstyle{morphism}=[fill=white, draw=black, shape=rectangle]
\tikzstyle{medium box}=[fill=white, draw=black, shape=rectangle, minimum width=0.7cm, minimum height=0.7cm]
\tikzstyle{large morphism}=[fill=white, draw=black, shape=rectangle, minimum width=1.7cm, minimum height=1cm]
\tikzstyle{bn}=[fill=black, draw=black, shape=circle, inner sep=1.5pt]
\tikzstyle{state}=[fill=white, draw=black, regular polygon, regular polygon sides=3, minimum width=0.8cm, shape border rotate=180, inner sep=0pt]
\tikzstyle{medium state}=[fill=white, draw=black, regular polygon, regular polygon sides=3, minimum width=1.3cm, inner sep=0pt, shape border rotate=180]
\tikzstyle{large state}=[fill=white, draw=black, regular polygon, regular polygon sides=3, minimum width=2.2cm, shape border rotate=180, inner sep=0pt]
\tikzstyle{wide state}=[fill=white, draw=black, shape=isosceles triangle, minimum width=0.8cm, shape border rotate=270, inner sep=1.4pt, minimum height=0.5cm, isosceles triangle apex angle=80]
\tikzstyle{wn}=[fill=white, draw=black, shape=circle, inner sep=1.5pt]
\tikzstyle{blue morphism}=[fill=white, draw={rgb,255: red,15; green,0; blue,150}, shape=rectangle, text={rgb,255: red,15; green,0; blue,150}, tikzit category=blue]
\tikzstyle{blue state}=[fill=white, draw={rgb,255: red,15; green,0; blue,150}, shape=circle, regular polygon, regular polygon sides=3, minimum width=0.8cm, shape border rotate=180, inner sep=0pt, text={rgb,255: red,15; green,0; blue,150}, tikzit category=blue]
\tikzstyle{blue node}=[fill={rgb,255: red,15; green,0; blue,150}, draw={rgb,255: red,15; green,0; blue,150}, shape=circle, tikzit category=blue, inner sep=1.5pt]
\tikzstyle{blue}=[text={rgb,255: red,15; green,0; blue,150}, tikzit draw={rgb,255: red,191; green,191; blue,191}, tikzit category=blue, tikzit fill=white, inner sep=0mm]
\tikzstyle{blue wide state}=[fill=white, draw={rgb,255: red,15; green,0; blue,150}, text={rgb,255: red,15; green,0; blue,150}, shape=isosceles triangle, minimum width=0.8cm, shape border rotate=270, inner sep=1.4pt, minimum height=0.5cm, isosceles triangle apex angle=80]
\tikzstyle{red node}=[fill={rgb,255: red,150; green,0; blue,2}, draw={rgb,255: red,150; green,0; blue,2}, shape=circle, inner sep=1.5pt]
\tikzstyle{Purple node}=[fill={rgb,255: red,120; green,0; blue,120}, draw={rgb,255: red,120; green,0; blue,120}, text={rgb,255: red,120; green,0; blue,120}, shape=circle, inner sep=1.5pt]
\tikzstyle{red}=[text={rgb,255: red,150; green,0; blue,2}, inner sep=0mm, tikzit fill=white, tikzit draw={rgb,255: red,191; green,191; blue,191}]
\tikzstyle{purple}=[text={rgb,255: red,150; green,0; blue,150}, inner sep=0mm, tikzit fill=white, tikzit draw={rgb,255: red,191; green,191; blue,191}]
\tikzstyle{white morphism}=[fill=white, draw=white, shape=rectangle, tikzit draw={rgb,255: red,139; green,139; blue,139}]
\tikzstyle{leak morphism}=[fill=white, draw={rgb,255: red,120; green,0; blue,85}, shape=rectangle, text={rgb,255: red,120; green,0; blue,85}, tikzit category=leak]
\tikzstyle{leak}=[text={rgb,255: red,120; green,0; blue,85}, inner sep=0mm, tikzit fill=white, tikzit draw={rgb,255: red,191; green,191; blue,191}, tikzit category=leak]
\tikzstyle{leak node}=[fill={rgb,255: red,120; green,0; blue,85}, draw={rgb,255: red,120; green,0; blue,85}, shape=circle, inner sep=1.5pt, tikzit category=leak]

\tikzstyle{arrow}=[->]
\tikzstyle{dashed box}=[-, dashed]
\tikzstyle{blue line}=[-, draw={rgb,255: red,15; green,0; blue,150}, tikzit category=blue]
\tikzstyle{red arrow}=[-, draw={rgb,255: red,150; green,0; blue,2}, tikzit category=red]
\tikzstyle{purple line}=[draw={rgb,255: red,120; green,0; blue,120}, >=stealth, shorten <=2pt, shorten >=2pt, -]
\tikzstyle{protected purple line}=[draw={rgb,255: red,120; green,0; blue,120}, >=stealth, shorten <=2pt, shorten >=2pt, preaction={line width=1.8pt, white, draw}, -]
\tikzstyle{mapsto}=[{|->}]
\tikzstyle{double wire}=[-, double]
\tikzstyle{protected}=[-, preaction={line width=1.8pt,white,draw}]
\tikzstyle{leak arrow}=[-, line join=round, decorate, decoration={snake, segment length=4, amplitude=0.75, pre=curveto, post=curveto, pre length=1pt, post length=1pt}]
\tikzstyle{protected leak arrow}=[-, line join=round, decorate, decoration={snake, segment length=4, amplitude=0.75, pre=curveto, post=curveto, pre length=1pt, post length=1pt}, preaction={line width=1.8pt, white, draw}]
\tikzstyle{hollow arrow}=[-, very thin, white, preaction={line width=0.7pt,draw={rgb,255: red,120; green,0; blue,85}}, tikzit category=leak, tikzit draw={rgb,255: red,150; green,0; blue,120}]
\tikzstyle{protected hollow arrow}=[-, very thin, white, preaction={line width=0.7pt,draw={rgb,255: red,120; green,0; blue,85},preaction={line width=2.1pt,white,draw}}, tikzit category=leak, tikzit draw={rgb,255: red,150; green,0; blue,120}]
\tikzstyle{curly brace}=[-, decorate, decoration={brace,amplitude=5pt}]

\usetikzlibrary{decorations.pathreplacing}	

\newcommand\tikzfigscaling{0.8}

\newcommand{\tikzfigscaled}[1]{\adjustbox{scale=\tikzfigscaling}{\input{#1.tikz}}}

\newtheorem{definition}{Definition}

\newtheorem{proposition}{Proposition}
\newtheorem{lemma}{Lemma}
\newtheorem{example}{Example}

\newcommand{\R}{\mathbb R}

\newcommand{\op}{\mathrm{op}}
\renewcommand{\det}{\mathrm{det}}

\newcommand{\cat}[1]{\mathbb{{#1}}}
\newcommand{\C}{\cat{C}}
\renewcommand{\P}{\cat{P}}
\renewcommand{\S}{\cat{S}}

\renewcommand{\sp}[1]{\mathbf{#1}}

\newcommand{\catname}[1]{\mathbf{{#1}}}
\newcommand{\set}{\catname{Set}}
\newcommand{\setmulti}{\catname{SetMulti}}
\newcommand{\finstoch}{\catname{FinStoch}}
\newcommand{\borelstoch}{\catname{BorelStoch}}
\newcommand{\gauss}{\catname{Gauss}}
\newcommand{\strongname}{\catname{StrongName}}
\newcommand{\nom}{\catname{Nom}}
\newcommand{\snom}{\catname{sNom}}
\newcommand{\inj}{\catname{FinInj}}
\newcommand{\surj}{\catname{Surj}}
\newcommand{\iso}{\catname{Isom}}
\newcommand{\coiso}{\catname{CoIsom}}
\newcommand{\con}{\catname{Con}}

\newcommand{\keyword}[1]{\mathrm{{#1}}}
\newcommand{\id}{\keyword{id}}
\newcommand{\cpy}{\keyword{copy}}
\newcommand{\del}{\keyword{del}}

\newcommand{\N}{\mathcal N}
\newcommand{\s}{\,|\,}
\newcommand{\ase}[1]{=_{#1}}

\newcommand{\reldot}{\otimes} 

\newcommand{\stat}[1]{\sp{#1}} 

\newcommand{\A}{\mathbb A}
\newcommand{\rep}[1]{\A^{\ast{#1}}}
\newcommand{\name}[2]{\langle {#1} \rangle {#2}}
\newcommand{\orb}{\mathrm{orb}}

\newcommand{\perm}{\mathrm{Perm}}
\newcommand{\supp}{\mathrm{supp}}

\newcommand{\psh}[1]{\mathrm{Psh}{#1}}
\newcommand{\sh}[1]{\mathrm{Sh}{#1}}
\newcommand{\disc}[1]{\underline{#1}}
\newcommand{\rv}{\mathrm{RE}}
\newcommand{\y}{\mathbf{y}}

\newcommand{\defeq}{\stackrel{\mathrm{def}}{=}}

\begin{document}

\title{Random Variables, Conditional Independence and Categories of Abstract Sample Spaces}

\author{\IEEEauthorblockN{Dario Stein}
\IEEEauthorblockA{\textit{iHub}, \textit{Radboud Universiteit Nijmegen}\\
Nijmegen, The Netherlands \\
ORCID 0009-0002-1445-4508}
}

\maketitle

\begin{abstract}
Two high-level "pictures" of probability theory have emerged: one that takes as central the notion of random variable, and one that focuses on distributions and probability channels (Markov kernels). While the channel-based picture has been successfully axiomatized, and widely generalized, using the notion of Markov category, the categorical semantics of the random variable picture remain less clear. Simpson's probability sheaves are a recent approach, in which probabilistic concepts like random variables are allowed vary over a site of sample spaces. Simpson has identified rich structure on these sites, most notably an abstract notion of conditional independence, and given examples ranging from probability over databases to nominal sets. 

We aim bring this development together with the generality and abstraction of Markov categories: We show that for any suitable Markov category, a category of sample spaces can be defined which satisfies Simpson's axioms, and that a theory of probability sheaves can be developed purely synthetically in this setting. We recover Simpson's examples in a uniform fashion from well-known Markov categories, and consider further generalizations. 
\end{abstract}


\section{Introduction}
Two pictures (or formalisms) of stochastic computation have emerged: The classical treatment (e.g. \cite{kallenberg1997foundations}) focuses on the notion of \emph{random variable}. Traditionally, one considers a \emph{sample space} $(\Omega,\mathcal E,P)$ consisting of a measurable space $(\Omega,\Sigma)$ and a probability measure $P : \mathcal E \to [0,1]$. A random variable is then a measurable function $\Omega \to V$ into a space of interest (typically $\R$). 

Another picture, prominent in theoretical computer science and physics, has taken the notion of \emph{probability channel} or \emph{Markov kernel} $\kappa : (X,\mathcal E) \leadsto (Y,\mathcal F)$ between measurable spaces as primitive. This is a parameterized probability distribution, associating to every $x \in X$ a probability measure $\kappa(x,-)$ on $Y$ in a measurable way. A channel between finite discrete spaces can be described as a stochastic matrix (a.k.a. conditional probability table). In this picture, it is studied how distributions and channels compose and decompose without postulating random variables. Distributions are a special case of channels with trivial input $X=1$. 

The two pictures have different strengths and weaknesses as they emphasize different notions. Equality in distribution and conditional distributions is natural to express using channels, while equality almost surely or conditional expectation are more naturally phrased using random variables. 

The channel-based picture is very amenable to the language of category theory. Recent efforts in categorical probability have succeeded in re-phrasing various probabilistic concepts in a categorical form (e.g. \cite{cho2019disintegration,fritz2023absolute,fritz2020infinite,fritz2021finetti,perrone2024convergence,ensarguet2023categorical,fritz2023d}). Notably, the abstract notion of \emph{Markov categories} \cite{fritz2019synthetic} has encapsulated important aspects of channel-based probability and opened them up to substantial generalizations. There are a wide variety of Markov categories capturing different structural models of channel-based probability, such as Gaussian probability, nondeterminism or fresh name generation. Channel-based probability also has close ties to the denotational semantics of probabilistic programs (e.g. \cite{staton2017commutative,stein2021structural}).

How to recover the random-variable picture from the channel-based one is less clear: One difficulty, noted by Tao \cite{tao}, is that the formal status sample space $\Omega$ is somewhat ephemeral. Often, this sample space is extended on-the-fly, and the (implicit) promise is that all meaningful constructions remain invariant under these extensions. Alex Simpson has proposed to model random variables using \emph{probability sheaves} over categories $\S$ in which the sample space varies \cite{simpson2017probability,simpson:atomic}. This is similar to fresh name generation and various forms of generativity in computer science, which use sheaves over `many worlds' to capture information available at different stages (the classical example is the Schanuel topos \cite{pitts2013nominal}). These sheaf-theoretic constructions guarantee equivariance, that is everything is consistent with respect to extension of sample spaces.

Simpson's work has identified rich structure on these categories $\S$ of sample spaces. Crucially, they carry atomic topologies \cite{simpson:atomic} and admit an axiomatic notion of conditional independence \cite{simpson:independence} called \emph{independent pullbacks}. He has given an array of example categories satisfying his assumptions, for example $\catname{FinProb}$ (discrete probability), $\catname{Prob}$ (Borel probability), $\surj$ (nondeterminism). It makes sense to think of these categories as \emph{abstract categories of sample spaces}.

The goal of this work is to connect Simpson's analysis with the Markov-categorical picture. We define, for any suitable Markov category $\C$, a category of sample spaces $\S(\C)$ and show that a theory of conditional independence and probability sheaves can be developed purely synthetically on these categories. Concretely,  
\begin{enumerate}
\item we recall probability spaces $\P(\C)$ and sample spaces $\S(\C)$ in Section~\ref{sec:ps_categories}
\item we define the independent pullback structure on $\S(\C)$, and prove that it satisfy the axioms of \cite{simpson:atomic} (Section~\ref{sec:independence_structure})
\item we recover Simpson's example categories as instances of our construction for the well-known Markov categories $\C=\catname{FinStoch},\catname{BorelStoch},\setmulti$ (Section~\ref{sec:examples}). While Simpson used a mix of abstract and model-specific arguments, our proofs rely solely on synthetic proofs and notions involving the category $\C$
\item we obtain new examples of sample spaces by analyzing Markov categories such as $\catname{Gauss}$ (Gaussian probability) and $\catname{StrongName}$ (fresh name generation). Their categories of sample spaces turn out to be equivalent to well-known and interesting examples: $\coiso$ (linear co-isometries) and $\catname{FinInj}^\op$ (opposite of finite sets and injections)
\item we begin developing the theory probability sheaves abstractly over our category of sample spaces $\S(\C)$ in Section~\ref{sec:sheaves}. This development generalizes the construction of the Schanuel topos for fresh name generation \cite{pitts2013nominal}. The relation between nominal techniques and stochastic independence has recently been of interest in probabilistic separation logics \cite{li:nominal}.
\end{enumerate}

The following concepts are of central importance in the context of our development:

\paragraph{Probability spaces} 

A probability space is the same object $(\Omega,\mathcal E,P)$ as a sample space, but the notion of morphism between them is different: A morphism of probability spaces is known as a \emph{coupling}, sometimes also \emph{kernel} or \emph{joint distribution} \cite{dahlqvist2019semantics,kozen2023joint}. Probability spaces are of course a ubiquitous notion in all of probability theory and functional analysis, but recently abstract properties of categories of probability spaces have been highlighted and used to great effect \cite{perrone2024convergence,ensarguet2023categorical,parzygnat2023axioms}.

\emph{Bayesian inversion} \cite{cho2019disintegration}, one could argue, is the most fundamental operation of machine learning, as it expresses the updating of our state of knowledge given a new observation. Probability spaces are a natural setting to study this operation, as Bayesian inversion forms a \emph{dagger functor} on this category, i.e. a contravariant involution.

Probability spaces $\P(\C)$ have been defined abstractly for Markov categories $\C$ \cite{fritz2019synthetic}. As sample spaces form a subcategory $\S(\C) \subseteq \P(\C)$, our work develops the theory of both categories in tandem. There is some interesting interplay between morphisms of probability spaces (couplings) and of sample spaces (extensions), which is absent in earlier work of Simpson.

\paragraph{Conditional independence}

Conditional independence is a crucial assumption for statistical modelling, and for reasoning about and optimizing probabilistic programs \cite{li:lilac}. There are various interrelated axiomatizations of conditional independence, most prominently semigraphoids \cite{pearl2022graphoids}, but also conditional products \cite{dawid1999conditional} or the independence structures considered by Simpson \cite{simpson:independence}. From a more generalized point of view, independence relations are pervasive to many contexts: stochastic independence, logical independence, and separation (freshness, separation logic). As was the case in \cite{simpson:independence}, all these aspects are modelled by our various examples. 

\section{Categories of Abstract Sample Spaces}\label{sec:ps_categories}

\noindent In this section, we recall the notion of \emph{Markov category}, which is a general framework for nondeterministic processes, such as probability, nondeterminism or fresh name generation. We then recall the derived notions of \emph{probability space} and \emph{sample space} over a suitable Markov category $\C$, which are our central objects of study. We recall abstract categorical characterizations of \emph{conditionals} and \emph{almost-sure equality} as auxiliary notions. Readers familiar with these notions may wish to skip to section \ref{sec:examples} which discusses and recovers Simpson's examples.  

\subsection{Recap: Markov Categories}

\noindent Markov categories capture some fundamental aspects of probabilistic computations: They can be composed in sequence (categorical composition $\circ$), in parallel (monoidal composition $\otimes$), and information can be copied and discarded by means of distinguished maps $\cpy_X : X \to X \otimes X$ and $\del_X : X \to I$. The copy maps are \emph{not} natural, which allows for differentiating between correlation (copying) and independence (re-computing). As is standard, we will use string diagrams \cite{selinger2011survey} alongside ordinary categorical composition to manipulate complex composites in a visually appealing form. Unless otherwise indicated, the material in this background section stems from \cite{fritz2019synthetic}. Recall that a symmetric monoidal category is \emph{semicartesian} if its unit $I$ is a terminal object. This endows every pair of objects $X,Y$ with canonical projections $X \xleftarrow{\pi_X} X \otimes Y \xrightarrow{\pi_Y} Y$.

\begin{definition}
A Markov category is a semicartesian symmetric monoidal category $(\C,\otimes,I)$ where each object is equipped with the structure of a commutative comonoid $(X,\cpy_X,\del_X)$ compatible with the monoidal structures (see Figure~\ref{fig:markov_axioms} for the full axioms). We render (iterated) copy and delete as follows
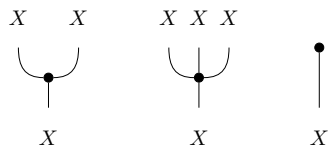
\[ \tikzfigscaled{markov_ops} \]
Note that terminality of $I$ implies that deletion is natural. 
\end{definition}

\noindent The prototypical examples of Markov categories come from probability theory; the formalism encompasses various ``flavors'', such as discrete, topologically continuous, measure-theoretic probability.

\begin{example}[Discrete probability]
The Markov category $\finstoch$ has as objects finite sets $X$, and as morphisms stochastic matrices $p \in \R^{Y \times X}$, meaning $p(y|x) \geq 0$ and
\[ \forall x \in X, \sum_y p(y|x) = 1 \]
Composition is matrix multiplication.
\end{example}

\begin{example}[Borel probability]
The category $\borelstoch$ has as objects standard Borel spaces $(X,\Sigma_X)$, and morphisms are \emph{Markov kernels}, i.e. maps $\kappa : X \times \Sigma_Y \to [0,1]$ such that $\kappa(x,A)$ is measurable in $x \in X$ and a probability measure in $A \in \Sigma_Y$. Composition is integration
\[ (\kappa \circ \tau)(x,A) = \int \kappa(y,A) \tau(x,\mathrm d y)\]
\end{example}

\noindent Certain subcategories of $\borelstoch$ admit a simple description without measure theory. A notable case is \emph{Gaussian probability}, which is the study of affine-linear maps with Gaussian (multivariate normal) noise.

\begin{example}[Gaussian probability]
The category $\gauss$ has as objects the spaces $\R^n$, and morphisms $\R^m \to \R^n$ are triples $(A,b,\Sigma)$ where $A \in \R^{m \times n}$, $b \in \R^n$ and $\Sigma \in \R^{n \times n}$ is positive semidefinite. The triple represents the probability kernel informally written $f(x) = Ax + \mathcal N(b,\Sigma)$. Composition is given by the rule
\[ (A,b,\Sigma) \circ (C,d,\Xi) = (AC, b+Ad,\Sigma + A\Xi A^T) \] 
\end{example}

\noindent Now for non-probabilistic examples:

\begin{example}[Nondeterminism]
The category $\setmulti$ has as objects sets $X$, and morphisms $X \to Y$ are left-total relations $R \subseteq X \times Y$, meaning $\forall x \exists y, (x,y) \in R$. Composition is relation composition.
\end{example}

\noindent We can view morphisms in $\setmulti$ as computations which may choose between one or more possible outputs nondeterministically. Another example due to \cite{stein2021structural} formalizes fresh name generation as a Markov category. 

\begin{example}[Fresh name generation]
The category $\strongname$ has as objects strong nominal sets $X$, and morphisms are (equivariant) Kleisli maps $X \to N(Y)$ for the name generation monad (a.k.a. free restriction set monad) $N$ of \cite[Section~9.5]{pitts2013nominal}. For example, there are exactly two Kleisli maps of type $\A \to N(\A)$, namely the identity $i(a) = \langle \rangle a$ and the fresh-name allocating function $f(a) = \langle b \rangle b \text{ for any } b \neq a$. Here the brackets $\langle b \rangle$ represent name binding. We present this category in more detail in the appendix (Section~\ref{sec:nom}).
\end{example}

\noindent As noted in \cite[Example~3]{fritz2019synthetic}, the Kleisli category of any commutative and affine monad will have the structure of a Markov category. Our previous examples arise from monads, namely the distribution monad $D$, Giry monad $\mathcal G$, nonempty powerset monad $\mathcal P^+$, and the name-generation monad $N$. $\gauss$ does not seem to be associated with any monad. 

\subsection{Probabilistic Notions in Markov categories}

\noindent Let $\C$ be a Markov category. We call morphisms $p : I \to X$ states. We use the abbreviation $\langle f,g \rangle = (f \otimes g) \circ \cpy$ for pairing of morphisms. A morphism $f : X \to Y$ is called \emph{deterministic} \cite[Definition~10.1]{fritz2019synthetic} if it commutes with copying
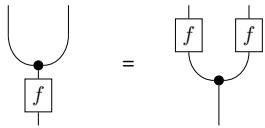
\[ \tikzfigscaled{f_det} \]

\noindent A Markov category $\C$ \emph{has conditionals} \cite[Definition~11.5]{fritz2019synthetic} if for every $f : A \to X \otimes Y$ there exists a factorization of the form
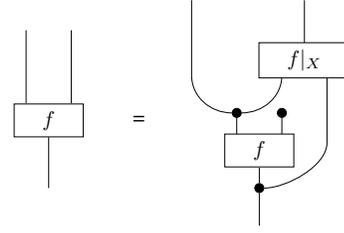
\[ \tikzfigscaled{f_cond} \]

\noindent A special case of conditionals is the \emph{Bayesian inverse}: For every $f : X \to Y$ and $p : I \to X$, there exists $f^\dagger_p : Y \to X$ such that
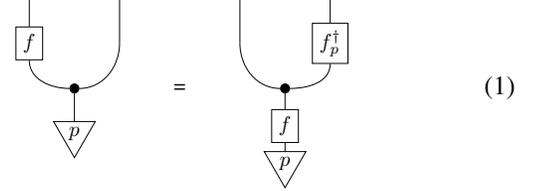
\begin{equation}
\tikzfigscaled{f_dagger} \label{eq:def_dagger}
\end{equation}

\noindent All running examples have conditionals (see~\cite{fritz2019synthetic}, the case for $\strongname$ is in the appendix). For $\borelstoch$, conditionals correspond precisely to \emph{regular conditional probabilities}.

We will now relativize Markov-categorical notions with respect to a state $p : I \to X$: Two morphisms $f, g : X \to Y$ are called \emph{$p$-almost surely equal}, written $f \ase p q$, if we have the following equality of states
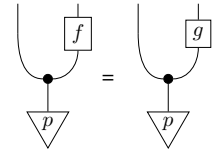
\[ \tikzfigscaled{f_ase} \]
\noindent This definition captures the usual mathematical meaning of almost sure equality in our examples. For example in $\finstoch$, we have $f \ase{p} g$ if $f(y|x) = g(y|x)$ for all $x,y$ with $p(x) > 0$. A morphism $f : X \to Y$ is called $p$-almost surely deterministic if
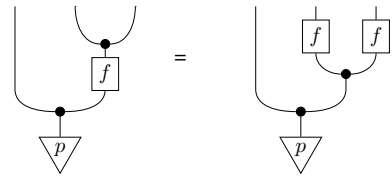
\[ \tikzfigscaled{f_as_det} \]

\noindent In a Markov category with conditionals, several important \emph{proof principles} are derivable \cite{fritz2023dilations}:
\begin{enumerate}
\item every isomorphism is deterministic
\item \emph{positivity}\footnote{The name is due to the fact that this property typically fails in models that include negative probabilities.}: this expresses that deterministic variables are independent of everything. if $f : A \to X \otimes Y$ has a deterministic marginal (e.g. $\pi_1 \circ f$ is deterministic) then $f$ is the product of its marginals $f = \langle \pi_1 \circ f, \pi_2 \circ f \rangle$
\item \emph{relative positivity}: if $f : A \to X \otimes Y$ has a $p$-almost surely deterministic marginal, then $f \ase{p} \langle \pi_X \circ f, \pi_2 \circ f \rangle$
\end{enumerate}

\subsection{Probability Spaces, Couplings, Sample Spaces}

\noindent Let $\C$ be a Markov category with conditionals. 

\begin{definition}[Probability spaces]
We define the category $\P(\C)$ of probability spaces as follows
\begin{enumerate}
	\item A \emph{probability space} in $\C$ is a pair $\sp \Omega = (\Omega,p)$ of an object $\Omega$ and a state $p : I \to \Omega$
	\item  A morphism of probability spaces $(\Omega,p) \to (\Omega',q)$ is an equivalence class $[f]_p$ of morphisms $f : \Omega \to \Omega'$, up to $p$-almost sure equality, which preserve the state, i.e. satisfy $f \circ p = q$.
	\item Composition is composition in $\C$ on representatives, which is well-defined because $[f]_{gp} \circ [g]_p = [f \circ g]_p$.
\end{enumerate}
\end{definition}

\noindent In abstract terms, almost sure equality defines a congruence relation on the slice category $I/\C$, and $\P(\C)$ the quotient under this congruence. By slight abuse of notation, we will simply write $f$ instead of $[f]$. The construction $\P(\C)$ is known under the names $\catname{ProbStoch}(\C)$, $\catname{PS}$, or $\mathcal S$ in \cite{fritz2019synthetic,ensarguet2023categorical,parzygnat2023axioms}. 

Central to our development is the subcategory $\S(\C)$ of $\P(\C)$ where morphisms are restricted to the almost surely deterministic ones. This category has not been named in the literature, and we will call it \emph{sample spaces} here. Thus while `probability space' and `sample space' are synonymous for objects $(\Omega,p)$, not every morphism of probability spaces is a morphism of sample spaces. 

\begin{definition}[Sample spaces]
The category of sample spaces is defined as the wide subcategory $\S(\C) \subseteq \P(\C)$ consisting of the morphisms $f : (\Omega,p) \to (\Omega',q)$ which are $p$-almost surely deterministic. 
\end{definition}

\noindent For brevity, when speaking about morphisms of probability spaces, we will shorten `$p$-almost surely deterministic' to `deterministic' (as the $p$ can be read off from the domain probability space). We will also refer to those morphisms as \emph{maps} of sample spaces, as opposed to \emph{channels} which may be nondeterministic (i.e. lie in $\P(\C)$). We will now establish some basic structural properties of the categories $\P(\C)$ and $\S(\C)$ which will be of constant use in what follows. 

\begin{proposition}[\!\!\!{\cite[Section~13]{fritz2019synthetic}}]
The category $\P(\C)$ is semicartesian monoidal with $(X,p) \otimes (Y,q) = (X \otimes Y, p \otimes q)$. Bayesian inversion $f \mapsto f^\dagger$ is a contravariant involutive functor on $\P(\C)$ making it a dagger category (e.g. \cite{karvonen2019way})
\end{proposition}

\noindent In particular, the chain rule for Bayesian inversion
\[ (g \circ f)^\dagger_p \,\ase{gfp}\, f^\dagger_p \circ g^\dagger_{fp} \]
simplifies to mere contravariance $(g \circ f)^\dagger = f^\dagger \circ g^\dagger$ when formulated in the category $\P(\C)$. We recall some terminology in dagger categories: A morphism $f : X \to Y$ is called
\begin{itemize}
\item an \emph{isometry} if $f^\dagger \circ f = \id_X$ 
\item a \emph{co-isometry} if $f \circ f^\dagger = \id_Y$
\item \emph{unitary} if it is both isometry and co-isometry (it is an isomorphism with $f^{-1} = f^\dagger$).
\end{itemize}
An equivalence of dagger categories (dagger equivalence) consists of dagger-preserving functors and natural isomorphisms $FG \cong 1$, $GF \cong 1$ with unitary components \cite{karvonen2019way}.

It is well-known that $\P(\C)$ is dagger equivalent to the category of \emph{couplings} (sometimes \emph{kernels}, \emph{joint distributions}) where morphisms $(X,p) \to (Y,q)$ are states $\gamma : I \to X \otimes Y$ with $\pi_1 \circ \gamma = p$ and $\pi_2 \circ \gamma = q$ \cite[Remark~12.10]{fritz2019synthetic}. The dagger on couplings is given by composing with the swap isomorphism $\gamma \mapsto \sigma_{X,Y} \circ \gamma$. 

\begin{proposition}\label{prop:det_coiso}
The following are equivalent for a morphism $f$ in $\P(\C)$
\begin{enumerate}
\item $f$ is deterministic
\item $f$ is a co-isometry, i.e. satisfies $f \circ f^\dagger = \id$
\item $f$ is split epic 
\end{enumerate}
Furthermore, if $f \circ g$ is deterministic, so is $f$. 
\end{proposition}
\begin{proof}
The equivalence of the first two statements is \cite[Proposition~2.5]{ensarguet2023categorical}, and co-isometries are split epic. The statement $f \circ g$ deterministic $\Rightarrow$ $f$ deterministic is precisely \cite[Remark~13.17]{fritz2019synthetic}. In particular, if $f$ is split epic then $f \circ g = \id$ is deterministic, so $f$ is deterministic.
\end{proof}

It follows that $\S(\C)$ can be identified as the subcategory of co-isometries of $\P(\C)$. By the last point, every isomorphism in $\P(\C)$ is unitary and deterministic, hence lies in $\S(\C)$: That means if two probability spaces are isomorphic, they are also isomorphic \emph{as sample spaces}. 

\begin{proposition}\label{prop:sc_monic}
The category $\S(\C)$ is semicartesian monoidal. Every morphism in $\S(\C)$ is epic. The canonical projections $\sp X \leftarrow \sp X \otimes \sp Y \rightarrow \sp Y$ are jointly monic. 
\end{proposition}

\begin{example}\label{ex:extension}
The projection map $\pi_1 : (X \otimes Y,p) \to (X,p_X)$ serves as our prototypical intuition for a map of sample spaces. We can think of it as describing a consistent \emph{extension} of the sample space $(X,p_X)$ to a larger sample space (this perspective is important in Section~\ref{sec:sheaves}).
\end{example}

\noindent In contrast to Proposition~\ref{prop:sc_monic}, the canonical projections are generally \emph{not} jointly monic in $\P(\C)$. This can however be remedied for channels that have a deterministic marginal:
 
\begin{lemma}\label{lem:relpos} 
In a commuting diagram in $\P(\C)$, 
\[\begin{tikzcd}
	&& {\sp \Omega} \\
	\\
	{\sp Y} && {\sp X \otimes \sp Y} && {\sp X}
	\arrow["g"', curve={height=12pt}, from=1-3, to=3-1]
	\arrow[shift right, from=1-3, to=3-3]
	\arrow["{\phi,\psi}", shift left, from=1-3, to=3-3]
	\arrow["f", curve={height=-12pt}, from=1-3, to=3-5]
	\arrow["{\pi_2}"', from=3-3, to=3-1]
	\arrow["{\pi_1}", from=3-3, to=3-5]
\end{tikzcd}\]
if either $f$ or $g$ is deterministic, then $\phi = \psi$. 
\end{lemma}
\begin{proof}
This is an application of relative positivity. The maps $\phi$ has a deterministic marginal $f$, hence it is almost surely equal to the tupling of its marginals, i.e. $\phi \ase{\stat{\Omega}} \langle f,g \rangle$. The same is true for $\psi$, $\phi \ase{\stat{\Omega}} \psi$. 
\end{proof}

\section{Examples}\label{sec:examples}

\noindent We will now recover concrete descriptions of the categories $\P(\C)$ and $\S(\C)$ for our example Markov categories, up to equivalence of (dagger) categories. An important aid is Proposition~\ref{prop:faithful}, which enables us to drop almost sure equivalence classes in many cases.

\begin{definition}
A probability space $(\Omega,p)$ is \emph{faithful} if for all $f,g : \Omega \to X$, $f \ase p g$ implies $f = g$. 
\end{definition}
\noindent We consider the full subcategories $\P_f(\C)$ and $\S_f(\C)$ whose objects are faithful probability spaces. Those are considerably easier to work with, as we no longer need to take equivalence classes of morphisms, and almost sure determinism coincides with plain determinism.

\begin{proposition}\label{prop:faithful}
The following are equivalent
\begin{enumerate}
\item every probability space is isomorphic to a faithful one
\item the inclusions $\P_f(\C) \to \P(\C)$ and $\S_f(\C) \to \S(\C)$ are part of an equivalence of (dagger) categories
\item every state $p : I \to X$ has a \emph{split support} in the sense of \cite{fritz2023absolute,stein2021structural}, i.e. there are morphisms $i : S \to X$, $\pi : X \to S$ such that
\begin{itemize}
\item $\pi i = \id_S$
\item $i \pi \ase{p} \id_X$
\item for all $f,g : X \to Y$, $f \ase{p} g \Leftrightarrow fi=gi$
\end{itemize}
\end{enumerate}
\end{proposition}
\begin{proof}
In Section~\ref{app:markov}.
\end{proof}

\noindent The `support-inclusion' $i : S \to X$ identifies the support of the distribution $p$, in the sense that $p$-almost sure equality can be tested by pulling back along $i$. A probability space is faithful if its distribution is supported on all of $X$. 

All our example categories except $\borelstoch$ are known to have split supports of states \cite{fritz2023absolute}; in those cases, we obtain the following descriptions:

\begin{proposition}[Discrete probability]\phantom{We have}
\begin{enumerate}
\item A probability space $(\Omega,p)$ in $\finstoch$ is faithful iff $p(\omega) > 0$ for all $\omega \in \Omega$
\item The category $\P(\finstoch)$ is dagger equivalent to the category of faithful probability spaces $(\Omega,p)$, where morphisms are stochastic matrices preserving the state. This was known as $\mathscr S^+(\catname{FinStoch})$ in \cite{parzygnat2024reversing}.
\item  The category $\S(\finstoch)$ is equivalent to $\catname{FinProb}$ of \cite{simpson:independence}: objects are faithful probability spaces, and morphisms $(\Omega,p) \to (\Omega',q)$ are (necessarily surjective) functions $f : \Omega \to \Omega'$ with
\[ q(\omega') = \sum_{\omega \in f^{-1}(\omega')} p(\omega) \]
\end{enumerate}
\end{proposition}

\begin{proposition}[Nondeterminism]\phantom{We have}
	\begin{enumerate}
		\item A sample space $(X,R)$ with $\emptyset \neq R \subseteq X$ in $\setmulti$ is faithful iff $R=X$. 
		\item The category $\P(\setmulti)$ is dagger equivalent to the category $\catname{TotRel}$ of nonempty sets and total relations. Dagger is the relational converse.
		\item The category $\S(\setmulti)$ is equivalent to the category $\surj$ of nonempty sets and surjections \cite{simpson:independence}.
	\end{enumerate}
\end{proposition}

\noindent For Gaussian probability, we obtain an elegant characterization of sample spaces in terms of the following concepts of linear algebra. Recall that a matrix $A \in \R^{n \times m}$ is called
\begin{enumerate}
\item an \emph{isometry} if $||Ax||=||x||$ for all $x \in \R^m$ where $||-||$ is Euclidean distance. Equivalently, $A^TA=I_m$.
\item a \emph{co-isometry} if $A^T$ is an isometry (i.e. $AA^T = I_n$)
\item a \emph{contraction} if $||Ax|| \leq ||x||$ for all $x \in \R^m$. 
\end{enumerate}
Isometries, co-isometries and contractions form respective subcategories $\iso,\coiso,\con$ of the category $\catname{Mat}$ whose objects are natural numbers and morphisms are matrices. Transposition defines an isomorphism $\iso^\op \cong \coiso$, and a $\dagger$-structure on $\con$.

\begin{proposition}[Gaussian probability]\label{prop:s_gauss} \phantom{We have}
\begin{enumerate}
\item A sample space $(\R^n,\mathcal N(\mu,\Sigma))$ in $\gauss$ is faithful iff $\Sigma$ has full rank, i.e. is positive definite. Every such sample space is isomorphic to a \emph{standard sample space} of the form $(\R^n, \mathcal N(0,I_n))$ by means of Cholesky decomposition, i.e. a factorization $\Sigma = LL^T$.
\item The category $\P(\gauss)$ is dagger equivalent to $\con$.
\item The category $\S(\gauss)$ is equivalent to $\coiso$. 
\end{enumerate}
\end{proposition}
\begin{proof}[Proof sketch]
Because of state preservation, a measure preserving Gaussian channel between standard sample spaces $(\R^m,\N(0,I_m)) \to (\R^n,\N(0,I_n))$ is a tuple $(A,0,\Sigma)$ with $AA^T + \Sigma = I_n$. This condition is equivalent to $A$ being a contraction, and the channel is deterministic iff $\Sigma = 0$, i.e. $AA^T = I_n$, meaning $A$ is a coisometry. The full proof is elaborated in Section~\ref{sec:gaussian}.
\end{proof}

\begin{proposition}[Fresh name generation]\phantom{We have}
\begin{enumerate}
\item A sample space $(X,W)$ in $\strongname$ consists of a strong nominal set $X$ and an orbit $W \subseteq X$. The sample space is faithful iff $W = X$
\item Each sample space is isomorphic to one of the form $(\rep n, \rep n)$ for $n \in \mathbb N$
\item The category $\S(\strongname)$ is equivalent to $\inj^\op$, the opposite category of finite sets and injections.
\end{enumerate}
\end{proposition}
\begin{proof}
Elaborated in Section~\ref{sec:nom}.
\end{proof}

\noindent We remark that in $\borelstoch$, not every sample space is isomorphic to a faithful one, so we don't obtain a simplified description via Proposition~\ref{prop:faithful}. We simply work with the definition of $\S(\borelstoch)$ as-is. We remark that $\S(\borelstoch)$ is equivalent to the category $\mathbb{SBP}_0$ considered by Simpson in \cite{simpson:atomic}. The latter category is defined slightly differently, and the two definition can be shown equivalent because $\borelstoch$ is \emph{a.s.-compatibly representable} in the sense of \cite{fritz2023representable}.

\section{The Independence Structure on Sample Spaces}\label{sec:independence_structure}

\noindent Now that we have introduced and characterized categories of sample spaces $\S(\C)$, we can study notions of conditional independence on them. We will recall Simpson's axioms for independent pullbacks, show that they apply to our categories $\S(\C)$, and recover the independence structures given in the examples in \cite{simpson:independence}. The independence structure will play a crucial role when using $\S(\C)$ as a site for probability sheaves in Section~\ref{sec:sheaves}.

\textbf{Notation:} In what follows, $\C$ is always a Markov category with conditionals. We will generally denote channels (morphisms in $\P(\C)$) by Greek letters $\phi,\psi$, and maps (i.e. deterministic channels, morphisms in the subcategory $\S(\C)$) by Latin letters $f,g$. The letter $\pi$ also always denotes a map. 


We now consider categories $\S$ equipped with a distinguished collection of commutative squares that are called ``independent``. A commuting square is an \emph{independent pullback} if it is independent, and it satisfies the universal property of a pullback with respect to other independent squares, i.e. whenever the outer kite is independent, there exists a unique mediating map 
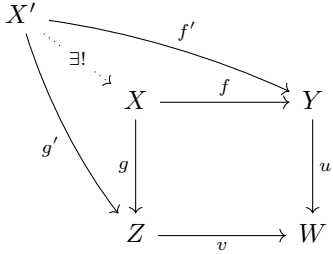
\[\begin{tikzcd}
	{X'} \\
	& X && Y \\
	\\
	& Z && W
	\arrow["{\exists!}"{description}, dotted, from=1-1, to=2-2]
	\arrow["{f'}", curve={height=-6pt}, from=1-1, to=2-4]
	\arrow["{g'}"', curve={height=6pt}, from=1-1, to=4-2]
	\arrow["f", from=2-2, to=2-4]
	\arrow["g"', from=2-2, to=4-2]
	\arrow["u", from=2-4, to=4-4]
	\arrow["v"', from=4-2, to=4-4]
\end{tikzcd}\]

\begin{definition}[\!\!\!{\cite{simpson:atomic}}]
An \emph{independent pullback structure} on a category $\S$ consists of a collection of commuting squares called independent, satisfying the following axioms
\begin{description}
\item[\textbf{(IP1)}] every square of the following form is independent
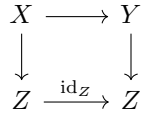
\[\begin{tikzcd}
	X & Y \\
	Z & Z
	\arrow[from=1-1, to=1-2]
	\arrow[from=1-1, to=2-1]
	\arrow[from=1-2, to=2-2]
	\arrow["{\id_Z}", from=2-1, to=2-2]
\end{tikzcd}\]
\item[\textbf{(IP2)}] if the left square is independent, so is the right
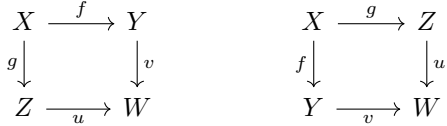
\[\begin{tikzcd}
	X & Y && X & Z \\
	Z & W && Y & W
	\arrow["f", from=1-1, to=1-2]
	\arrow["g"', from=1-1, to=2-1]
	\arrow["v", from=1-2, to=2-2]
	\arrow["g", from=1-4, to=1-5]
	\arrow["f"', from=1-4, to=2-4]
	\arrow["u", from=1-5, to=2-5]
	\arrow["u"', from=2-1, to=2-2]
	\arrow["v"', from=2-4, to=2-5]
\end{tikzcd}\]
\item[\textbf{(IP3)}] if (A) and (B) are independent, then so is (AB)
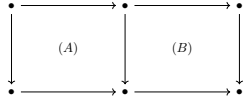
\[ 
\adjustbox{scale=0.5}{%
	\begin{tikzcd}
	\bullet && \bullet && \bullet \\
	& {(A)} && {(B)} \\
	\bullet && \bullet && \bullet
	\arrow[from=1-1, to=1-3]
	\arrow[from=1-1, to=3-1]
	\arrow[from=1-3, to=1-5]
	\arrow[from=1-3, to=3-3]
	\arrow[from=1-5, to=3-5]
	\arrow[from=3-1, to=3-3]
	\arrow[from=3-3, to=3-5]
\end{tikzcd}}\]
\item[\textbf{(IP4)}] If (AB) is independent and (B) is an independent \emph{pullback}, then (A) is independent
\item[\textbf{(IP5)}] Every cospan $Y \xrightarrow{u} W \xleftarrow{v} Z$ has a completion to an independent pullback
\end{description}
\end{definition}

\noindent Our central example for a category with an independent pullback structure will be the category of sample spaces $\S(\C)$. We will define the notion of independence now, but postpone the verification of the axioms (IP1)-(IP5) to Section~\ref{subsec:ip} until we have developed some theory around independent squares.

\begin{definition}
We call a commutative square in $\S(\C)$
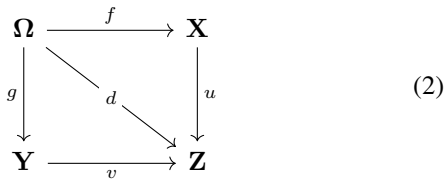
\begin{equation}\begin{tikzcd}
	{\sp \Omega} && {\sp X} \\
	\\
	{\sp Y} && {\sp Z}
	\arrow["f", from=1-1, to=1-3]
	\arrow["g"', from=1-1, to=3-1]
	\arrow["d"{description}, from=1-1, to=3-3]
	\arrow["u", from=1-3, to=3-3]
	\arrow["v"', from=3-1, to=3-3]
\end{tikzcd} \label{eq:square} \end{equation}
\emph{independent} if the maps $f,g$ are conditionally independent given $d$ in the sense of \cite[12.1]{fritz2019synthetic}, written $f \bot g \s d$. That means there exist channels $\phi : Z \to X$ and $\psi : Z \to Y$ in $\C$ such that
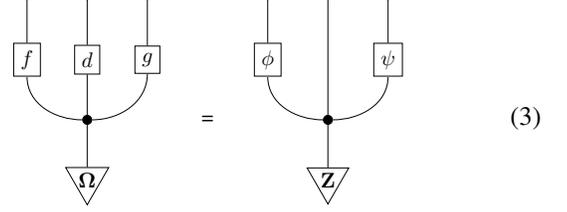
\begin{equation} \tikzfigscaled{indep_square} \label{eq:def_indep_square} \end{equation}
\end{definition}

\noindent We can simplify this condition considerably as follows:

\begin{lemma}\label{lem:indep}
For any commutative square as in \eqref{eq:square}, the following composites (I)-(VI) are equal
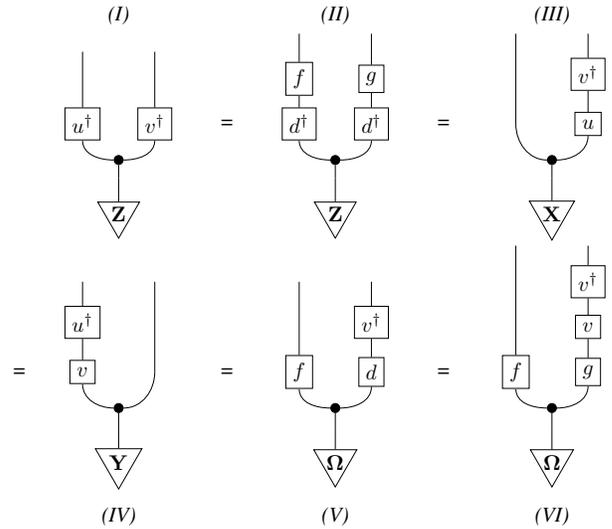
\begin{equation*} \tikzfigscaled{indep_characterization} \end{equation*}
for arbitrary choices of Bayesian inverses. Furthermore, the square \eqref{eq:square} is independent if and only if the joint state
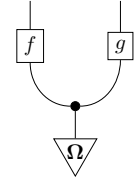
\[ \tikzfigscaled{indep_lhs} \]
is equal to any (equivalently: all) of the composites (I)-(VI).
\end{lemma}
\begin{proof}
The equality of (I)-(VI) is straightforward calculation (see Section~\ref{app:markov}). We prove that independence is equivalent to the equation $\langle f, g \rangle \circ p_\Omega = (I)$. 
\begin{enumerate}
\item Assume the square is independent, with channels $\phi,\psi$ witnessing equation \eqref{eq:def_indep_square}. By marginalization, one sees that $\phi$ is a choice of Bayesian inverse $u^\dagger$ (and $\psi$ is $v^\dagger$). By marginalizing the middle wire, we obtain as desired
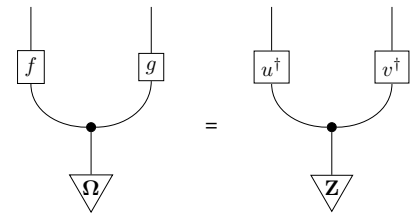
\begin{equation*} \tikzfigscaled{indep_square_simplified} \label{eq:I} \end{equation*}
\item Conversely, assuming equation \eqref{eq:I}, we construct the following factorization
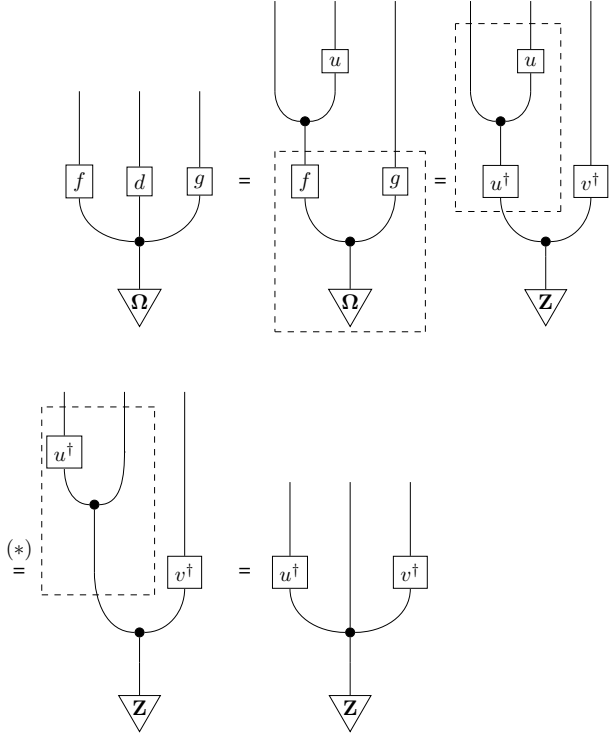
\[ \tikzfigscaled{indep_square_cond_lemma} \]
where the equations used are determinism of $f$, the hypothesis, and the step $(*)$ asserting that $\langle \id_X, u \rangle u^\dagger \ase{\stat Z} \langle u^\dagger, \id_Z \rangle$ by relative positivity.\qedhere
\end{enumerate}
\end{proof}

\noindent Matthew Di Meglio and Paolo Perrone have been working on a way to characterize independence in a purely dagger-categorical setup\footnote{Personal communication.}. We include their characterization here and provide a proof for reference.
\begin{proposition}\label{prop:dagger_characterization}
A commutative square \eqref{eq:square} in $\S(\C)$ is independent if and only if the following equation holds in $\P(\C)$
\begin{equation} gf^\dagger = v^\dagger u \label{eq:dagger} \end{equation}
\end{proposition}
\begin{proof}
Assume \eqref{eq:dagger}, then using the definition of Bayesian inverse, we have
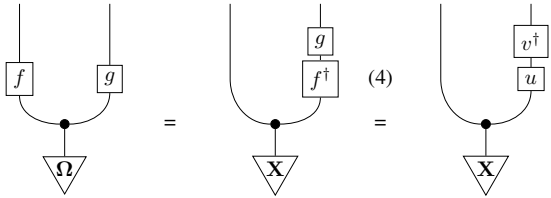
\[ \tikzfigscaled{dagger_lemma} \]
which establishes independence by Lemma~\ref{lem:indep} (III). The same argument shows that for an independent square, we have $gf^\dagger \ase{\stat X} v^\dagger u$ and hence equality in the quotient $\P(\C)$.  
\end{proof}

\begin{example}
Our synthetic definition of independence recovers the concrete ones defined in \cite{simpson:independence}. Spelling this out
\begin{enumerate}
\item in the probabilistic examples $\catname{FinProb}$, $\S(\gauss)$ and $\S(\borelstoch)$, independence this is the usual notion of conditional independence of random variables \cite[Examples~5.1,5.2]{simpson:independence}
\item for $\surj$, \eqref{eq:square} is independent if for all $x \in X, y \in Y$ with $u(x) = v(y)$ there exists $\omega \in \Omega$ with $f(\omega) = x$ and $g(\omega) = y$. This is, the square is a weak pullback in $\set$. This notion is related to \emph{variation independence} in database theory \cite{dawid2001separoids,simpson:independence}. 
\item following Proposition~\ref{prop:dagger_characterization}, a commuting square of matrices in $\coiso$ is independent if it satisfies $GF^T = V^TU$.
\end{enumerate}
\end{example}

\noindent At last, we observe that independent squares are pushouts, which we will return to in Section~\ref{sec:sheaves}.
\begin{proposition}\label{prop:pushout}
Independent squares are pushouts in $\S(\C)$.
\end{proposition}
\begin{proof}
Let the top square be independent; we need to show that a unique map $k$ exists.
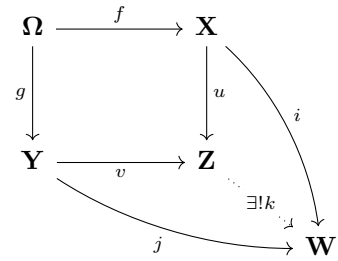
\[\begin{tikzcd}
	{\sp \Omega} && {\sp X} \\
	\\
	{\sp Y} && {\sp Z} \\
	&&& {\sp{W}}
	\arrow["f", from=1-1, to=1-3]
	\arrow["g"', from=1-1, to=3-1]
	\arrow["u", from=1-3, to=3-3]
	\arrow["i", curve={height=-12pt}, from=1-3, to=4-4]
	\arrow["v"', from=3-1, to=3-3]
	\arrow["j"', curve={height=12pt}, from=3-1, to=4-4]
	\arrow["{\exists! k}"{description}, dotted, from=3-3, to=4-4]
\end{tikzcd}\]
It suffices to show that $k=i \circ u^\dagger = j \circ v^\dagger$ holds and makes the diagram commute. This completes the proof, because any such $k$ is automatically unique (because $u$ is epic) and deterministic (by Proposition~\ref{prop:det_coiso}). Using Lemma~\ref{lem:indep} we have
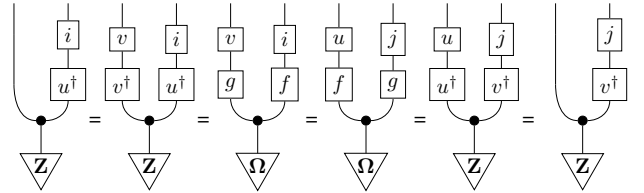
\[ \tikzfigscaled{pushout_1} \]
We now show that $k \circ u = i$, namely
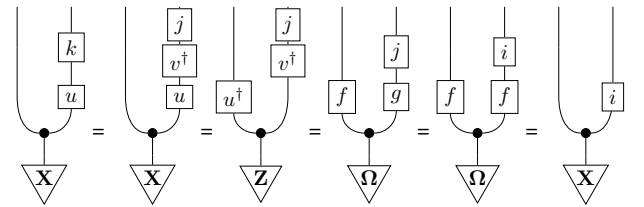
\[ \tikzfigscaled{pushout_2} \]
where we use $k=j \circ v^\dagger$, Bayesian inversion, Lemma~\ref{lem:indep} and determinism of $f$, respectively. The equation $k \circ v = j$ is shown analogously by instantiating $k=i \circ u^\dagger$. 
\end{proof}

\subsection{Independent Pullbacks}

\noindent We will now show that the independence structure thus defined has independent pullbacks, which are given by the relative product construction.

\begin{definition}
Given a cospan $\sp{X_1} \xrightarrow{u_1} \sp Y \xleftarrow{u_2} \sp{X_2}$ of sample spaces, its \emph{relative product} 
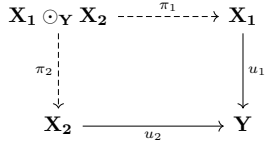
\begin{equation}
\adjustbox{scale=\tikzfigscaling}{\begin{tikzcd}
		{\sp{X_1} \odot_{\sp Y} \sp{X_2}} && {\sp{X_1}} \\
		\\
		{\sp{X_2}} && {\sp{Y}}
		\arrow["{\pi_1}", dashed, from=1-1, to=1-3]
		\arrow["{\pi_2}"', dashed, from=1-1, to=3-1]
		\arrow["{u_1}", from=1-3, to=3-3]
		\arrow["{u_2}"', from=3-1, to=3-3]
	\end{tikzcd}} \label{eq:rel_prod} \end{equation} is defined as the sample space $(X_1 \otimes X_2, \rho)$ where 
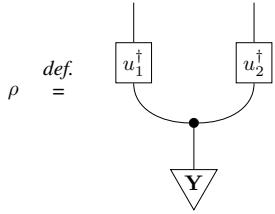
\[ \tikzfigscaled{rel_prod} \]
\end{definition}

\noindent The relative product is closely related to the notion of \emph{conditional product} as defined by Dawid and Studen\'y \cite{dawid1999conditional}, and studied in Markov categories by Fritz \cite[Definition~12.8]{fritz2019synthetic}. The conditional product is a state of type $X_1 \otimes Y \otimes X_2$, from which the relative product is obtained by marginalization over $Y$. Our use of the relative product generalizes its use in \cite{simpson:independence}.

\begin{proposition}
The relative product square \eqref{eq:rel_prod} commutes and is independent.
\end{proposition}
\begin{proof}
We need to show $u_1\pi_1 \ase{\rho} u_2\pi_2$, which upon some simplification and rearrangement requires us to show that
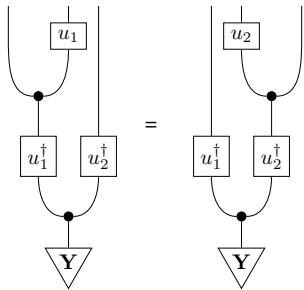
\[ \tikzfigscaled{rel_prod_commutes_2}\]
Using the definition of the Bayesian inverse and determinism of $u_1$, we transform the left hand side as follows
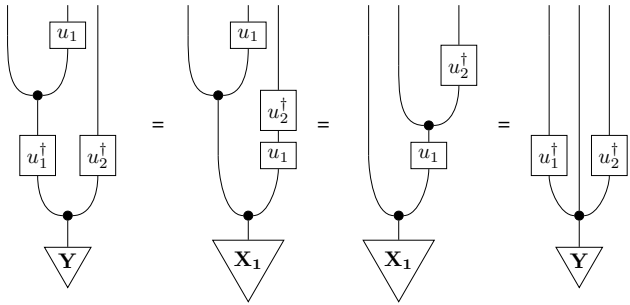
\[ \tikzfigscaled{rel_prod_commutes_3}\]
The analogous transformation is possible for the right-hand side, proving the desired equality. To show independence of the square, we establish form (I) of Lemma~\ref{lem:indep}.
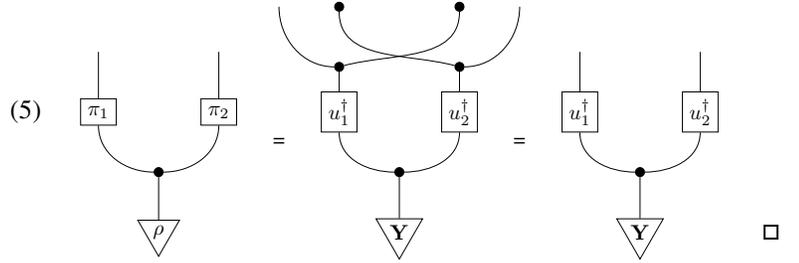
\[ \tikzfigscaled{rel_prod_indep} \]
\end{proof}

\begin{proposition}
The relative product square is an independent pullback.
\end{proposition}
\begin{proof}
Assume that  in the following diagram the outer kite is independent. We claim there exists a unique mediating map $h$ as shown. 
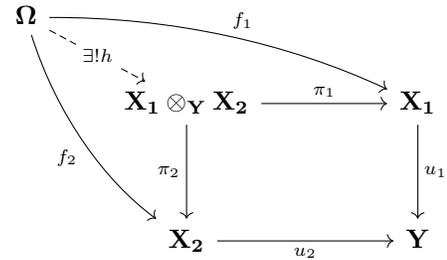
\[\begin{tikzcd}
	{\sp{\Omega}} \\
	& {\sp{X_1} \reldot_{\sp Y} \sp{X_2}} && {\sp{X_1}} \\
	\\
	& {\sp{X_2}} && {\sp Y}
	\arrow["{\exists!h}"{description}, dashed, from=1-1, to=2-2]
	\arrow["{f_1}", curve={height=-12pt}, from=1-1, to=2-4]
	\arrow["{f_2}"', curve={height=12pt}, from=1-1, to=4-2]
	\arrow["{\pi_1}", from=2-2, to=2-4]
	\arrow["{\pi_2}"', from=2-2, to=4-2]
	\arrow["{u_1}", from=2-4, to=4-4]
	\arrow["{u_2}"', from=4-2, to=4-4]
\end{tikzcd}\]
Uniqueness is clear because the projections are jointly monic on deterministic maps (Lemma~\ref{lem:relpos}). In fact, we are forced to choose $h = \langle f_1, f_2 \rangle$ to make the diagram commute. It remains to show that this choice of $h$ is a valid morphism in $\S$ (i.e. measure-preserving): $\langle f_1, f_2 \rangle \circ p_\Omega = \rho$. But this is precisely criterion (I) for the independence in Lemma~\ref{lem:indep}. 
\end{proof}

\subsection{Weakness and Descent}

\noindent In Simpson's work, the categories $\S$ and their independence structures were studied in isolation. In our setting, because $\S(\C)$ is derived from Markov category, we can study the interplay of independent pullbacks not just with deterministic maps but also general (nondeterministic) channels from the supercategory $\P(\C)$. This reveals a certain analogy between arbitrary independent squares and \emph{weak pullbacks}. These results are essential stepping stones in the proof of the independence axiom (IP4), as well as of independent interest as they generalize Simpson's notion of \emph{descent} for an independence structure \cite{simpson:atomic}.

\begin{proposition}[Weak independent pullbacks]
Consider two independent squares in $\S(\C)$ over the same cospan 
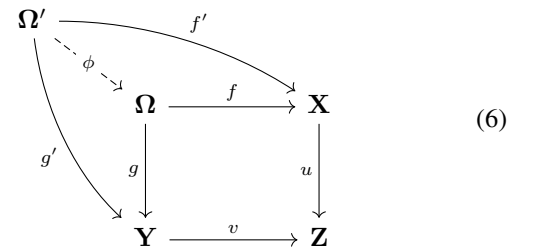
\begin{equation}\begin{tikzcd}
	{\sp{\Omega'}} \\
	& {\sp \Omega} && {\sp X} \\
	\\
	& {\sp Y} && {\sp Z}
	\arrow["\phi"{description}, dashed, from=1-1, to=2-2]
	\arrow["{f'}", curve={height=-12pt}, from=1-1, to=2-4]
	\arrow["{g'}"', curve={height=12pt}, from=1-1, to=4-2]
	\arrow["f", from=2-2, to=2-4]
	\arrow["g"', from=2-2, to=4-2]
	\arrow["u"', from=2-4, to=4-4]
	\arrow["v", from=4-2, to=4-4]
\end{tikzcd}\label{eq:two_squares} \end{equation}
Then there exists a measure-preserving \emph{channel} $\phi : \Omega' \to \Omega$ (not necessarily deterministic, nor unique) making the diagram commute in $\P(\C)$. 
\end{proposition}
\begin{proof}
It suffices to construct a map $\sp X \reldot_{\sp Z} \sp Y \to \sp{\Omega}$. The general case can then be solved using the composite $\sp{\Omega'} \to \sp X \reldot_{\sp Z} \sp Y \to \sp{\Omega}$. Consider the diagram
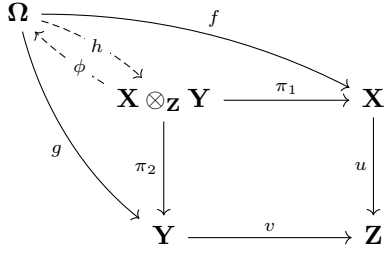
\[\begin{tikzcd}
	{\sp \Omega} \\
	& {\sp X \reldot_{\sp Z} \sp Y} && {\sp X} \\
	\\
	& {\sp Y} && {\sp Z}
	\arrow["h"{description}, curve={height=-6pt}, dashed, from=1-1, to=2-2]
	\arrow["f", curve={height=-12pt}, from=1-1, to=2-4]
	\arrow["g"', curve={height=12pt}, from=1-1, to=4-2]
	\arrow["\phi"{description}, curve={height=-6pt}, dashed, from=2-2, to=1-1]
	\arrow["{\pi_1}", from=2-2, to=2-4]
	\arrow["{\pi_2}"', from=2-2, to=4-2]
	\arrow["u"', from=2-4, to=4-4]
	\arrow["v", from=4-2, to=4-4]
\end{tikzcd}\]
where $h$ be the mediating unique map into the independent pullback, and define the channel $\phi = h^\dagger$ as its Bayesian inverse. This is measure-preserving and satisfies $h\phi=\id$ by determinism. Therefore, $\phi$ makes the diagram commute as $f\phi = \pi_1 h \phi = \pi_1$ and $g\phi = \pi_2 h \phi = \pi_2$.
\end{proof}

\begin{proposition}[Nondeterministic Descent]\label{prop:nondet_descent}
Consider a diagram of commutative squares in $\S(\C)$ as follows,
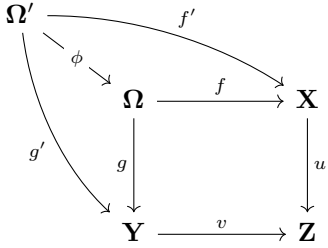
\[\begin{tikzcd}
	{\sp{\Omega'}} \\
	& {\sp \Omega} && {\sp X} \\
	\\
	& {\sp Y} && {\sp Z}
	\arrow["\phi"{description}, from=1-1, to=2-2]
	\arrow["{f'}", curve={height=-12pt}, from=1-1, to=2-4]
	\arrow["{g'}"', curve={height=12pt}, from=1-1, to=4-2]
	\arrow["f", from=2-2, to=2-4]
	\arrow["g"', from=2-2, to=4-2]
	\arrow["u", from=2-4, to=4-4]
	\arrow["v", from=4-2, to=4-4]
\end{tikzcd}\]
Then the outer kite is independent if and only if the inner square is. The statement remains true if the mediating map $\phi$ is allowed to be a channel, i.e. lie in $\P(\C)$. 
\end{proposition}
\begin{proof}
The map $\langle f, g \rangle \circ \phi$ has the deterministic marginals $f',g'$, so by Lemma~\ref{lem:relpos} it is the product of its marginals, $\langle f, g \rangle \circ \phi \ase{\stat{\Omega'}} \langle f',g'\rangle$. Hence we have an equality
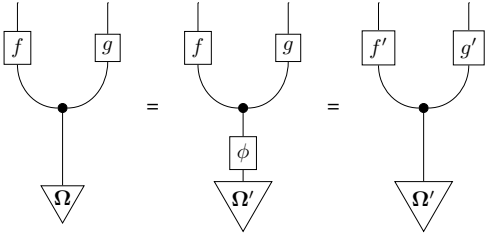
\[ \tikzfigscaled{indep_kite} \]
This concludes the proof by \ref{lem:indep} Criterion (I).
\end{proof}

\noindent This proposition strengthens the descent property of \cite{simpson:atomic} which restricts $\phi$ to maps instead of arbitrary channels. Similarly, we can also strengthen the universal property of the independent pullback to quantify over mediating channels:
\begin{proposition}\label{prop:strong_uniqueness}
In the situation \eqref{eq:two_squares}, the inner square is an independent pullback if and only if whenever the outer kite is independent, then there exists a unique mediating \emph{channel} $\phi$.
\end{proposition}
\begin{proof}
If the inner square is an independent pullback, there exists a mediating map $h$. To show uniqueness among \emph{channels}, replace without loss of generality $\sp{\Omega}$ with $\sp{X} \reldot_{\sp Z} \sp{Y}$. By \ref{lem:relpos}, any other mediating channel $\phi$ must be deterministic, hence equal to $h$ by ordinary uniqueness for independent pullbacks. 

Conversely, let $h : \sp \Omega \to \sp{X} \reldot_{\sp Z} \sp Y$ be the mediating map into the relative product. Then $hh^\dagger = \id$ by determinism, and the following diagram commutes
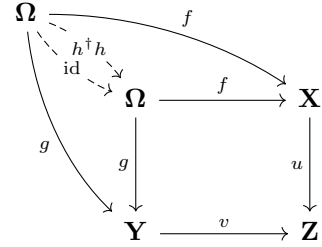
\[\begin{tikzcd}
	{\sp \Omega} \\
	& {\sp \Omega} && {\sp X} \\
	\\
	& {\sp Y} && {\sp Z}
	\arrow["{h^\dagger h}"{description}, curve={height=-6pt}, dashed, from=1-1, to=2-2]
	\arrow["\id"{description}, curve={height=6pt}, dashed, from=1-1, to=2-2]
	\arrow["f", curve={height=-12pt}, from=1-1, to=2-4]
	\arrow["g"', curve={height=12pt}, from=1-1, to=4-2]
	\arrow["f", from=2-2, to=2-4]
	\arrow["g"', from=2-2, to=4-2]
	\arrow["u"', from=2-4, to=4-4]
	\arrow["v", from=4-2, to=4-4]
\end{tikzcd}\]
By uniqueness of mediating channels, we have $h^\dagger h = \id_{\sp \Omega}$, i.e. $h$ is an isomorphism. 
\end{proof}

\subsection{Verification of the Independence Axioms}\label{subsec:ip}

\noindent We can now proceed to verify Simpson's axioms (IP1)-(IP5) for our notion of independence on the category $\S(\C)$. Most of these are straightforward, with (IP4) taking the most work.
 
\begin{proposition}[IP1]
Every square in $\S(\C)$ of the following form is independent
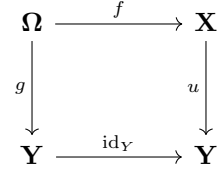
\[\begin{tikzcd}
	{\sp \Omega} && {\sp X} \\
	\\
	{\sp Y} && {\sp Y}
	\arrow["f", from=1-1, to=1-3]
	\arrow["g"', from=1-1, to=3-1]
	\arrow["u"', from=1-3, to=3-3]
	\arrow["{\id_Y}", from=3-1, to=3-3]
\end{tikzcd}\]
\end{proposition}
\begin{proof}
Using determinism of $f$ and criterion (III) of Lemma~\ref{lem:indep}:
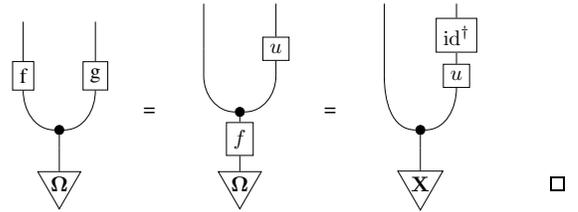
\[ \tikzfigscaled{ip1} \]
\end{proof}

\begin{proposition}[IP2]
In $\S(\C)$, if the left square is independent, so is the right one.
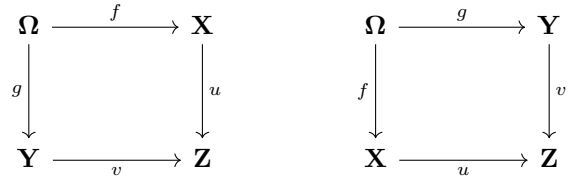
\[\begin{tikzcd}
	{\sp \Omega} && {\sp X} && {\sp \Omega} && {\sp Y} \\
	\\
	{\sp Y} && {\sp Z} && {\sp X} && {\sp Z}
	\arrow["f", from=1-1, to=1-3]
	\arrow["g"', from=1-1, to=3-1]
	\arrow["u", from=1-3, to=3-3]
	\arrow["g", from=1-5, to=1-7]
	\arrow["f"', from=1-5, to=3-5]
	\arrow["v", from=1-7, to=3-7]
	\arrow["v"', from=3-1, to=3-3]
	\arrow["u"', from=3-5, to=3-7]
\end{tikzcd}\]
\end{proposition}
\begin{proof}
Immediate from commutativity of the copy maps. 
\end{proof}

\begin{proposition}[IP3]
If (A) and (B) are independent composable squares in $\S(\C)$, as in
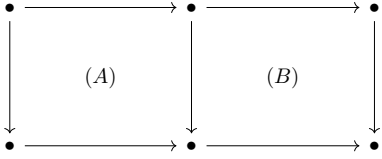
\begin{equation}\adjustbox{scale=\tikzfigscaling}{\begin{tikzcd}
	\bullet && \bullet && \bullet \\
	& {(A)} && {(B)} \\
	\bullet && \bullet && \bullet
	\arrow[from=1-1, to=1-3]
	\arrow[from=1-3, to=1-5]
	\arrow[from=1-5, to=3-5]
	\arrow[from=3-1, to=3-3]
	\arrow[from=3-3, to=3-5]
	\arrow[from=1-3, to=3-3]
	\arrow[from=1-1, to=3-1] \label{eq:composite_squares}
\end{tikzcd}}\end{equation}
then the composite square (AB) is independent.
\end{proposition}
\begin{proof}
We label the squares as follows
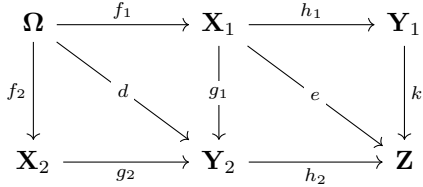
\[\begin{tikzcd}
	{\sp \Omega} && {\sp X_1} && {\sp Y_1} \\
	\\
	{\sp X_2} && {\sp Y_2} && {\sp Z}
	\arrow["{f_1}", from=1-1, to=1-3]
	\arrow["{f_2}"', from=1-1, to=3-1]
	\arrow["d"{description}, from=1-1, to=3-3]
	\arrow["{h_1}", from=1-3, to=1-5]
	\arrow["{g_1}"{description}, from=1-3, to=3-3]
	\arrow["e"{description}, from=1-3, to=3-5]
	\arrow["k", from=1-5, to=3-5]
	\arrow["{g_2}"', from=3-1, to=3-3]
	\arrow["{h_2}"', from=3-3, to=3-5]
\end{tikzcd}\]
We can now verify criterion (I) by applying Lemma~\ref{lem:indep} to the independent squares (A) and (B)
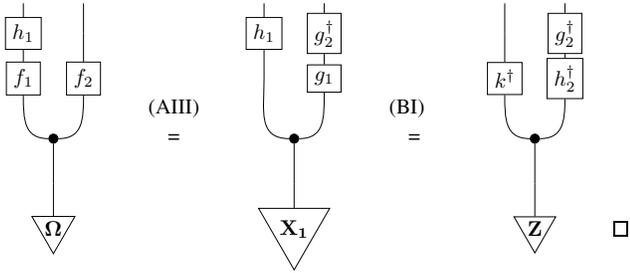
\[ \tikzfigscaled{ip3} \]
\end{proof}

The following property (IP4) requires the most machinery to prove. Proposition~\ref{prop:nondet_descent} enables us to adapt the usual proof strategy for the following variant of the pullback lemma: If $(AB)$ is a weak pullback and $(B)$ a pullback, then $(A)$ is a weak pullback. 

\begin{proposition}[IP4]
In the situation \eqref{eq:composite_squares}, if the composite rectangle (AB) is independent and (B) is an independent pullback, then (A) is independent. 
\end{proposition}
\begin{proof}
By Proposition~\ref{prop:nondet_descent}, it suffices to construct a mediating channel $\phi$ into (A) from an arbitrary independent square $\sp{\Omega'}$ .
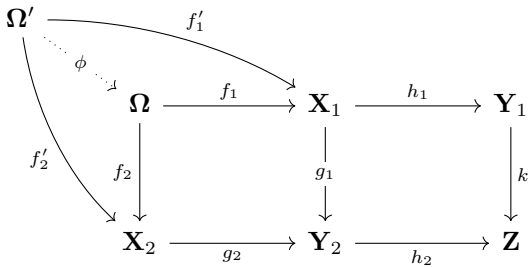
\[\begin{tikzcd}
	{\sp {\Omega'}} \\
	& {\sp \Omega} && {\sp X_1} && {\sp Y_1} \\
	\\
	& {\sp X_2} && {\sp Y_2} && {\sp Z}
	\arrow["\phi"{description}, dotted, from=1-1, to=2-2]
	\arrow["{f_1'}", curve={height=-12pt}, from=1-1, to=2-4]
	\arrow["{f_2'}"', curve={height=12pt}, from=1-1, to=4-2]
	\arrow["{f_1}", from=2-2, to=2-4]
	\arrow["{f_2}"', from=2-2, to=4-2]
	\arrow["{h_1}", from=2-4, to=2-6]
	\arrow["{g_1}"{description}, from=2-4, to=4-4]
	\arrow["k", from=2-6, to=4-6]
	\arrow["{g_2}"', from=4-2, to=4-4]
	\arrow["{h_2}"', from=4-4, to=4-6]
\end{tikzcd}\]
\noindent Because the outer square $(AB)$ is independent, there exists a mediating channel $\phi$ which makes the following diagram commute
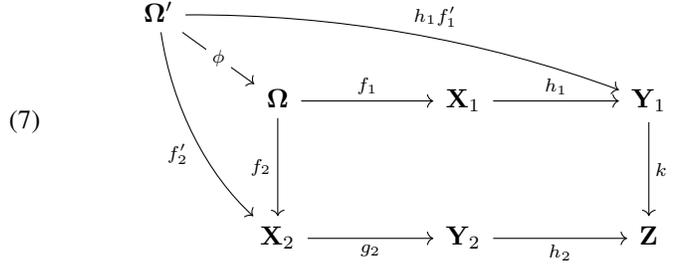
\[\begin{tikzcd}
	{\sp {\Omega'}} \\
	& {\sp \Omega} && {\sp X_1} && {\sp Y_1} \\
	\\
	& {\sp X_2} && {\sp Y_2} && {\sp Z}
	\arrow["\phi"{description}, from=1-1, to=2-2]
	\arrow["{h_1f_1'}", curve={height=-12pt}, from=1-1, to=2-6]
	\arrow["{f_2'}"', curve={height=12pt}, from=1-1, to=4-2]
	\arrow["{f_1}", from=2-2, to=2-4]
	\arrow["{f_2}"', from=2-2, to=4-2]
	\arrow["{h_1}", from=2-4, to=2-6]
	\arrow["k", from=2-6, to=4-6]
	\arrow["{g_2}"', from=4-2, to=4-4]
	\arrow["{h_2}"', from=4-4, to=4-6]
\end{tikzcd}\]
We claim that $\phi$ also mediates the smaller square (A), i.e. additionally satisfies $f_1\phi = f_1'$. For this, note that the following diagram commutes with two mediating maps
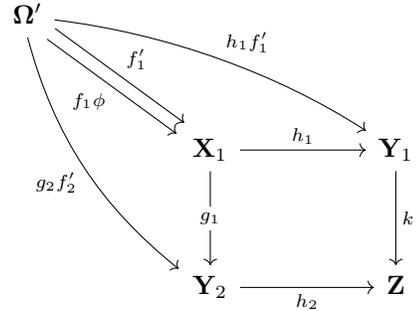
\[\begin{tikzcd}
	{\sp \Omega'} \\
	\\
	&& {\sp X_1} && {\sp Y_1} \\
	\\
	&& {\sp Y_2} && {\sp Z}
	\arrow["{f_1'}", shift left, from=1-1, to=3-3]
	\arrow["{f_1\phi}"', shift right, from=1-1, to=3-3]
	\arrow["{h_1f_1'}", curve={height=-12pt}, from=1-1, to=3-5]
	\arrow["{g_2f_2'}"', shift right, curve={height=12pt}, from=1-1, to=5-3]
	\arrow["{h_1}", from=3-3, to=3-5]
	\arrow["{g_1}"{description}, from=3-3, to=5-3]
	\arrow["k", from=3-5, to=5-5]
	\arrow["{h_2}"', from=5-3, to=5-5]
\end{tikzcd}\]
By Proposition~\ref{prop:strong_uniqueness}, we conclude that any two mediating channels must be equal, hence $f_1\phi = f_1'$ as desired. 
\end{proof}

\begin{proposition}[IP5]
Every cospan $\sp{X_1} \to \sp Y \leftarrow \sp{Y_2}$ in $\S(\C)$ admits a completion to an independent pullback.
\end{proposition}
\begin{proof}
Given by the relative product construction. 
\end{proof}

\section{Probability Sheaves and Random Variables}\label{sec:sheaves}

We can now begin to recreate Simpson's treatment of probability sheaves \cite{simpson2017probability,simpson:atomic} in our abstract setup. As before, $\C$ is a Markov category with conditionals, and $\S(\C)$ its category of sample spaces.

A \emph{probability presheaf} is a functor $P : \S(\C)^\op \to \set$. Concretely such a presheaf $P$ consists of a family of sets $P(\sp \Omega)$ indexed over sample spaces. If $\pi : \sp{\Omega'} \to \sp{\Omega}$ is a morphism in $\S(\C)$ and $x \in P(\sp \Omega)$ is an element, then we denote the functorial action of $P$ as 
\[ x \cdot \pi \defeq P(\pi)(x) \in F(\sp \Omega') \]
In accordance with Example~\ref{ex:extension}, we can see this action as an extension of the element $x$ to the larger sample space $\sp{\Omega'}$. A morphism of presheaves $f : P \to Q$ is a natural transformation; the naturality condition says that $f$ is equivariant with respect to the extension action: for all $x \in P(\sp \Omega)$ and $\pi : \sp{\Omega'} \to \sp{\Omega}$ we have $f_{\sp{\Omega'}}(x \cdot \pi) = f_{\sp \Omega}(x) \cdot \pi$. We write $\psh(\S(\C))$ for the topos of probability presheaves. 

\subsection{Presheaf of Random Elements}

\noindent For any set $A$, we write $\disc{A}$ for the constant presheaf with $\disc{A}(\sp{\Omega}) = A$. Unlike constant presheaves, the notion of \emph{random element} depends on the underlying sample space $\sp{\Omega}$. We formalize this as follows:

\begin{definition}
For each object $V$ of $\C$, we define a presheaf $\rv(V) : \S(\C)^\op \to \set$ of \emph{random elements} valued in $V$, as 
\begin{enumerate}
\item $\rv(V)(\Omega,p) \defeq \{ [X] : \Omega \to V \emph{ $p$-a.s. det } \}$ consists of $p$-almost sure equivalence classes of $p$-almost-surely deterministic morphisms in $\C$. 
\item the extension action is given by precomposition on representatives. If $X : \Omega \to V$ and $\pi : \sp{\Omega'} \to \sp{\Omega}$, then
\[ [X] \cdot \pi \defeq [X \circ \pi] \]
\end{enumerate}
\end{definition}

\begin{example}
For our example categories, the presheaves of random elements take the following simple forms
\begin{enumerate}
\item for $V \in \finstoch$, $\Omega \in \catname{FinProb}$, 
\[ \rv(V)(\sp{\Omega}) \cong \{ X : \Omega \to V \text{ any function } \} \]
\item for $V \in \setmulti$, $\Omega \in \surj$,
\[ \rv(V)(\Omega) \cong \{ X : \Omega \to V \text{ any function } \} \] 
\item for $V \in \borelstoch$, $\sp \Omega \in \S(\borelstoch)$,
\[ \rv(V)(\sp{\Omega}) \cong \{ X : \Omega \to V \text{ measurable } \} \]
\item for $\R^n \in \gauss$, $\R^m \in \coiso$,
\[ \rv(\R^n)(\R^m) \cong \{ X : \R^m \to \R^n \text{ affine-linear } \} \] 
\item for $V \in \strongname$, $K \in \catname{FinInj}^\op$, 
\[ \rv(V)(K) \cong \{ f : \rep{K} \to V \text{ equivariant } \}\]
\end{enumerate}
\end{example}

\begin{example}
We have a well-defined natural transformation $\mathrm{Law} : \rv(V) \to \disc{\C(I,V)}$ which assigns a random element $[X] : \Omega \to V$ to its law, by $\mathrm{Law}_{\sp \Omega}([X]) = X \circ p_{\Omega}$.
\end{example}

\noindent The presheaf $\rv(V)$ is closely related to the representable presheaves on $\S(\C)$: If $\sp{\Omega'}$ is a sample space, then $\y(\sp{\Omega'}) = \S(\C)(-,\sp{\Omega'})$ is a sub-presheaf of $\rv(\Omega')$ consisting of those random elements whose law is equal to $p_{\Omega'}$. Conversely, $\rv(V)$ is isomorphic to the coproduct of representables
\[ \rv(V) \cong \bigsqcup_{p : I \to V} \S(\C)(-, (V,p)) \]

\begin{proposition}
The random element construction defines a functor $\rv : \C_\det \to \psh(\S(\C))$ where $\C_\det \subseteq \C$ is the subcategory of deterministic morphisms. For $h : V \to W$ deterministic we define the natural transformation
\[ \rv(h) : \rv(V) \to \rv(W), \quad \rv(h)([X])_{\sp \Omega} \defeq [h \circ X] \]
\noindent Furthermore the tensor product on $\C_{\det}$ is a cartesian product, and the $\rv$ functor preserves it
\[ \rv(U \otimes V) \cong \rv(U) \times \rv(V) \]
\end{proposition}

\subsection{Sheaf Conditions}

\noindent In this section, we will show that the presheaves $\rv(V)$ and $\S(\C)(-,\sp{\Omega})$ are always \emph{sheaves} with respect to the \emph{atomic topology} on $\S(\C)$. This makes the Grothendieck topos of atomic sheaves $\sh(\S(\C))$ a natural setting for categorical probability. Sheaf conditions guarantee a well-behaved interplay between the values that a presheaf $P$ takes on different sample spaces $\sp{\Omega}$, and are intimately related with the independence structure (Propositions~\ref{prop:sheaf_pullback}, \ref{prop:sheaf_pushout}). Simpson has extensively studied the logical structure of atomic sheaves and their relationship with independence in \cite{simpson:atomic}.

Here, we will only introduce what is strictly needed about atomic topologies, following \cite{simpson:atomic}. For a general introduction to sheaf toposes, we refer to \cite{maclane2012sheaves}. The atomic topology is the Grothendieck topology where every singleton family $\{ \sp{\Omega'} \xrightarrow{\pi} \sp{\Omega} \}$ is covering. For this topology to be well-defined, one requires the \emph{right Ore condition}: every cospan can be completed to a commuting square. We have established this for $\S(\C)$ in (IP5). 

Let $P : \S(\C)^\op \to \set$ be a presheaf and $\pi : \sp{\Omega'} \to \sp{\Omega}$ be map in $\S(\C)$. An element $y \in P(\sp{\Omega'})$ is called $\pi$-invariant if for any parallel pair of maps $\rho, \rho' : \sp{\Omega''} \to \sp{\Omega'}$ with $\pi \circ \rho = \pi \circ \rho'$, we have $y \cdot \rho = y \cdot \rho'$.

\begin{definition}
A presheaf $P \in \psh(\S(\C))$ is \emph{separated} if for all $\pi : \sp{\Omega'} \to \sp{\Omega}$ and $x,y \in P(\sp{\Omega})$, if $x \cdot \pi = y \cdot \pi$ then $x=y$. 
\end{definition}

\begin{definition}
A presheaf $P \in \psh(\S(\C))$ is an \emph{atomic sheaf} if for every map $\pi : \sp{\Omega'} \to \sp{\Omega}$ and every $\pi$-invariant $y \in P(\sp{\Omega'})$, there is a unique $x \in P(\sp{\Omega})$ with $y = x \cdot \pi$. 
\end{definition}

\begin{proposition}
The presheaves $\rv(V)$ and $\y(\sp{\Omega'})$ are separated. 
\end{proposition}
\begin{proof}
Let $X,Y : \sp{\Omega} \to V$ be almost surely deterministic and let $\pi : \sp{\Omega'} \to \sp{\Omega}$ be an extension such that $X \circ \pi \ase{\sp{\Omega'}} Y \circ \pi$. By almost-sure determinism of $\pi$, we reason $X \ase{\sp{\Omega}} Y$ as
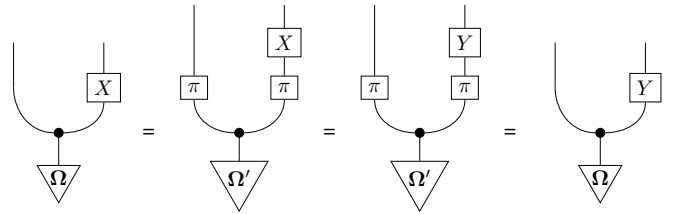
\[ \tikzfigscaled{separation} \]
The presheaf $\y(\sp{\Omega'})$ is separated as a sub-presheaf of the separated presheaf $\rv(\Omega')$. 
\end{proof}

\noindent The following helpful characterization relates $\pi$-invariant random elements to the notion of \emph{conditional expectation operators}. In categorical probability theory, conditional expectations can be identified with the composites $e = \pi^\dagger \circ \pi$ induced by almost-surely deterministic maps $\pi$ \cite{ensarguet2023categorical,perrone2024convergence}. The channel $e$ is a \emph{dagger idempotent}, meaning $e \circ e = e$ and $e^\dagger = e$. 

\begin{proposition}\label{prop:invariant}
Let $\pi : \sp{\Omega'} \to \sp{\Omega}$. A random element $Y \in \rv(V)(\sp{\Omega'})$ is $\pi$-invariant if and only if it satisfies $Y \ase{\sp{\Omega'}} Y \circ e$, where $e = \pi^\dagger \circ \pi$ is the conditional expectation associated with $\pi$.
\end{proposition}
\begin{proof}
Necessity is clear; if $Y \ase{\sp{\Omega'}} Y \circ \pi^\dagger \circ \pi$ and $\rho,\rho'$ are given with $\pi \circ \rho \ase{\sp{\Omega''}} \pi \circ \rho'$, then $Y \circ \rho \ase{\sp{\Omega''}} Y \circ \pi^\dagger \circ \pi \circ \rho \ase{\sp{\Omega''}} Y \circ \pi^\dagger \circ \pi \circ \rho' \ase{\sp{\Omega''}} Y \circ \rho'$, showing that $Y$ is $\pi$-invariant as an element of $\rv(V)$. For sufficiency, we apply $\pi$-invariance to the pair of projections $\rho_1, \rho_2 : (\Omega' \otimes \Omega',\psi) \to (\Omega',q)$, where $\psi = \langle \id, e \rangle \circ p_{\Omega'}$. 
\begin{enumerate}
\item $\rho_1,\rho_2$ are measure-preserving, because $e$ is an endomorphism on $\sp{\Omega'} \to \sp{\Omega'}$, hence $e \circ p_{\Omega'} = p_{\Omega'}$. 
\item we need to check that $\pi \rho_1 \ase \psi \pi \rho_2$. This simplifies to verifying that
	\[ \tikzfigscaled{glue_proj_inv} \]
	For this, it suffices to note that $\langle \id, \pi \rangle \circ e \ase{\sp{\Omega'}} \langle e, \pi \rangle$ by relative positivity. This is related to $e$ being a \emph{strong idempotent} in the sense of \cite[Definition~4.11]{fritz2023absolute}.
\item from $\pi$-invariance, we now obtain that $Y \circ \rho_1 \ase{\psi} Y \circ \rho_2$, which means
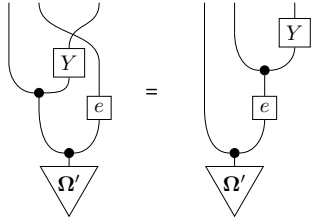
\[ \tikzfigscaled{glue_y_inv} \]
\noindent Marginalizing the middle wire gives $Y \ase{\sp{\Omega'}} Y \circ e$ as desired. \qedhere
\end{enumerate}
\end{proof}

\begin{proposition}
The presheaves $\rv(V)$ and $\y(\sp{V})$ are atomic sheaves. 
\end{proposition}
\begin{proof}
Let $Y : \sp{\Omega'} \to V$ be $\pi$-invariant for $\pi : \sp{\Omega'} \to \sp{\Omega}$; we claim that there exists a unique $X : \sp{\Omega} \to V$ with $Y=X \cdot \pi$.
Uniqueness follows from separation. For existence, we attempt the definition $X = Y \circ \pi^\dagger$. By Proposition~\ref{prop:invariant}, we have $Y = X \cdot \pi$ as desired, and $X$ is almost surely deterministic because $X\pi$ is (Proposition~\ref{prop:det_coiso}).
\end{proof}

We can compare our direct proof of the sheaf property two general propositions about presheaves on a site $\S$ with an independent pullback structure \cite{simpson:atomic}. They illustrate a deep relationship between the independence structure and the atomic topology.

\begin{proposition}[\!\!{\cite[Theorem~6.4]{simpson:atomic}}]\label{prop:sheaf_pullback}
The following are equivalent for a presheaf $P : \S^\op \to \set$
\begin{enumerate}
\item $P$ is an atomic sheaf
\item $P$ maps independent squares in $\S$ to pullbacks in $\set$.
\end{enumerate}
\end{proposition}

\begin{proposition}[\!\!\!{\cite[Corollary~6.6]{simpson:atomic}}]\label{prop:sheaf_pushout}
The following are equivalent 
\begin{enumerate}
\item representable presheaves are atomic sheaves
\item independent squares in $\S$ are pushouts.
\end{enumerate}
\end{proposition}

\noindent As we have shown that independent squares in $\S(\C)$ are pushouts (Proposition~\ref{prop:pushout}), this gives another way of showing that representable presheaves on $\S(\C)$ are indeed atomic sheaves. Our development has now come full circle. Starting with the theory of fresh name generation $\strongname$, the category of sample spaces $\S(\strongname)$ is equivalent to $\inj^\op$. Its atomic sheaf topos $\sh(\inj^\op)$ is the Schanuel topos, which is again equivalent to the category $\nom$ of nominal sets \cite{pitts2013nominal}.

\section{Conclusion and Future Work}

\noindent We have replicated substantial parts of Simpson's development on independence and probability sheaves in the synthetic setting of a sample spaces over a Markov category. We have phrased everything in terms of abstract notions and proof principles, without relying on specific details of the individual models. We have recovered known examples from probability theory, nondeterminism and fresh name generation in this setting, as well generalized the theory to novel ones such as Gaussian probability.

This line of work is only at the beginning, and several directions have been left for future work: It will be interesting to study the intrinsic logical notions of atomic sheaf toposes (supports, atomic equivalence etc. \cite{simpson:atomic}) in our setting. Similarly, a treatment of random variables featuring a notion of (conditional) expectation is future work. 

A monad $\mathcal M$ can be defined on probability sheaves which models the allocation of fresh random variables, defined in \cite{simpson2017probability} as
\[ \mathcal MP(\sp{\Omega}) = \int^{\sp{\Omega'}} \S(\Omega',\Omega) \times P(\Omega') \]
This monad is a direct generalization of the name generation monad $N$ on the Schanuel topos. Connections between probabilistic separation logic and nominal techniques are an active area of research \cite{li:lilac,li:nominal}. Given our characterization $\S(\gauss) \cong \coiso$, the category of atomic sheaves on isometries $\sh(\coiso) \cong [\iso,\set]_{\mathrm{at}}$ seems to be of interest as a linear-algebraic generalization of the Schanuel topos that seems tightly related to Gaussian probability. 

Lastly, we may wonder how the categories $\S$ and $\P$ and their relationships may be characterized abstractly, and how far such a pair of categories is from arising from a Markov category $\C$ with conditionals. We conjecture that the Kleisli category of the monad $\mathcal M$ on probability sheaves will be a Markov category into which $\C$ embeds, though conditionals are no longer guaranteed for that larger category. \\ 

\noindent\textbf{Acknowledgements:} I am grateful for the feedback and fruitful discussions about this work with many people, particularly Chris Heunen, Matthew Di Meglio, Paolo Perrone and Alex Simpson. A preliminary version of this work was presented at ItaCa Fest 2024.

\pagebreak

\bibliographystyle{plain}
\bibliography{main}

\section{Appendix: Proofs}\label{app:markov}

\begin{figure*}[th!]
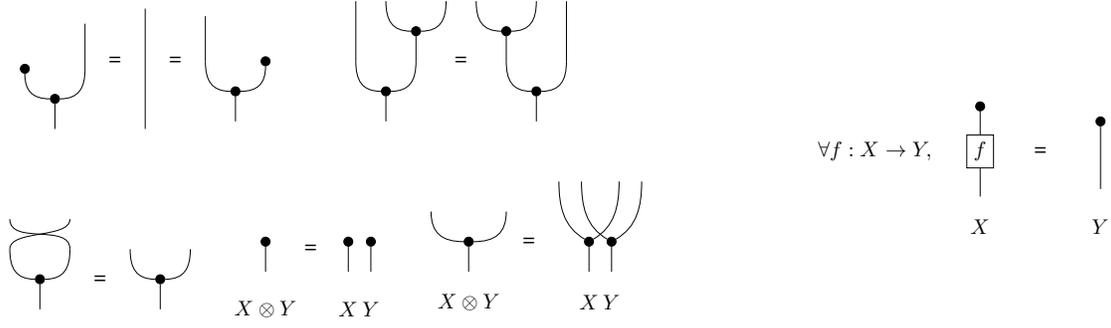

\begin{center}
$\vcenter{\hbox{\tikzfigscaled{markov_ax}}}$
\end{center}
\caption{The axioms for Markov categories}
\label{fig:markov_axioms}.
\end{figure*}

\begin{proof}[Proof of Proposition~\ref{prop:faithful}]
The first two points are equivalent by the usual characterization of equivalence of categories by means of an essentially surjective, full and faithful functor. For an equivalence of dagger categories, certain isomorphisms need to be unitary, but this is automatic in our setting (Proposition~\ref{prop:det_coiso}). It remains to show the equivalence of the first and third point. Take a probability space $(X,p)$ and choose an isomorphism 
	\begin{equation}\begin{tikzcd}
			{(X,p)} && {(S,\sigma)}
			\arrow["\pi", shift left, from=1-1, to=1-3]
			\arrow["i", shift left, from=1-3, to=1-1]
		\end{tikzcd} \label{eq:support_iso} \end{equation}
	with a faithful $(S,\sigma)$. We claim that this is a split support.
	\begin{itemize}
		\item we have $\pi i \ase{\sigma} \id_S$ because $\pi,i$ are inverses, but $\sigma$ is faithful, so $\pi i = \id_S$
		\item we have $i \pi \ase{p} \id_X$ because $\pi,i$ are inverses
		\item assume $fi = gi$, then $f \ase{p} g$ because
		\[ f = f \circ \id_X \ase{p} fi\pi = gi\pi \ase{p} g \circ \id_X = g \]
		\item conversely if $f \ase{p} g$, then by precomposition $fi \ase{\sigma} gi$ so $fi=gi$ by faithfulness of $\sigma$
	\end{itemize}
	Conversely, if $p$ has a split support, then \eqref{eq:support_iso} is an isomorphism for $\sigma = \pi p$. It remains to argue why $\sigma$ is faithful: For any $f, g : S \to Y$ with $f \ase{\sigma} g$, we have $f \pi \ase{p} g\pi$ by precomposition. By assumption on the support this implies $f\pi i = g \pi i$, but $\pi i = \id_S$ so $f=g$. 
\end{proof}

\begin{proof}[Proof of Lemma~\ref{lem:indep}]
We verify the equalities of the composites (I)-(VI). By the characteristic property of Bayesian inversion, we have
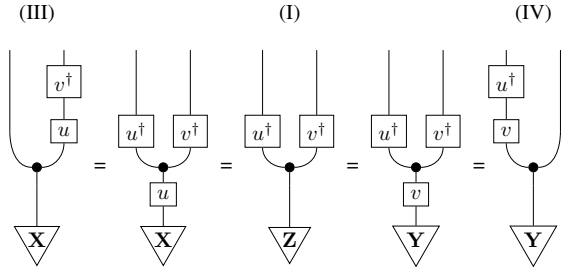
\[ \tikzfigscaled{indep_characterization_proof_1} \]
Because $f$ and $g$ are co-isometries, we have
\begin{align*}
u^\dagger &= ff^\dagger u^\dagger = f(uf)^\dagger = fd^\dagger \\
v^\dagger &= gg^\dagger v^\dagger = g(vg)^\dagger = gd^\dagger
\end{align*} 
showing $(I) = (II)$. Again, by determinism of $f$, we have
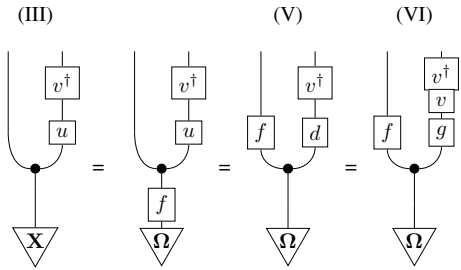
\[ \tikzfigscaled{indep_characterization_proof_2} \]
\end{proof} 

\section{Appendix: Example Categories}

In this appendix, we spell out definitions and examples in more detail. Independent squares will be labelled
\begin{equation}\begin{tikzcd}
	\Omega && X \\
	\\
	Y && Z
	\arrow["f", from=1-1, to=1-3]
	\arrow["g"', from=1-1, to=3-1]
	\arrow["u"', from=1-3, to=3-3]
	\arrow["v", from=3-1, to=3-3]
\end{tikzcd}\label{eq:square_appendix} \end{equation}

\subsection{Gaussian Probability}\label{sec:gaussian}

\noindent Given a Gaussian sample space, we can replace it up to isomorphism by a simpler sample space in two stages: first by a faithful one, and then by a \emph{standard sample space} which has covariance matrix $I_k$. 

Let $\mathcal N(\mu,\Sigma)$ be a Gaussian distribution which is supported on the subspace $S = \mu + \mathrm{im}(\Sigma)$. Let $k=\mathrm{dim}(S) = \mathrm{rank}(\Sigma)$ be its dimension, and choose an affine isomorphism $i : \R^k \to S$. Then $i$ is a split support inclusion as shown in \cite[III.10]{stein2021compositional}. A sample space $(\R^n, \mathcal N(\mu,\Sigma))$ is therefore faithful if and only if $\Sigma$ has full rank (i.e. is positive definite).

Assume now that $\Sigma$ has full rank; we call a \emph{standard sample space} one equipped with a standard normal distribution, i.e. $(\R^n,\mathcal N(0,I_n))$. 
By Cholesky decomposition, there is a decomposition $\Sigma = LL^T$ where $L \in \R^{n \times n}$ is invertible. Thus, the affine map $f(x) = Lx + \mu$ defines an isomorphism of sample spaces
\[ (\R^n,\mathcal N(0, I_n)) \xrightarrow{f} (\R^n,\mathcal N(\mu,\Sigma)) \] 
This shows that every sample space in $\S(\gauss)$ is isomorphic to a standard sample space. 

A channel in $\P(\gauss)$ between standard sample spaces 
\[ f : (\R^m,\N(0,I_m)) \to (\R^n,\N(0,I_n)) \]
is necessarily of the form $f(x) = Ax + \N(0,\Sigma)$ where measure preservation imposes the condition
\begin{equation} AA^T + \Sigma = I_n \label{eq:posdef} \end{equation} Because we can read off $\Sigma$ from \eqref{eq:posdef}, we can identify such channels with matrices $A \in \R^{n \times m}$ such that $I_n - AA^T$ is positive semidefinite. 

We are left to show that this condition is equivalent to $A$ being a contraction. This is a well-known fact, but we'll add a proof for reference. Firstly, note that $A$ is a contraction iff its operator norm is $||A|| \leq 1$, and $||A|| = ||A^T||$, thus $A$ is a contraction iff its transpose is. Now note that for all $y \in \R^n$,
\begin{align*}
0 &\leq \langle y,(I_n-AA^T)y\rangle \\
\langle A^Ty,A^Ty \rangle &\leq \langle y,y\rangle \\
||A^Ty|| &\leq ||y||
\end{align*}
hence $I_n - AA^T$ is positive semidefinite iff $A^T$ is a contraction. 

To compute the Bayesian inverse of $f$, we apply an ansatz $f^\dagger(y) = By + \N(0,\Phi)$ and solve equation \eqref{eq:def_dagger}, which reads
\[ \begin{pmatrix}
I_m & A^T \\
A & AA^T + \Sigma
\end{pmatrix} = \begin{pmatrix}
BB^T + \Phi & B \\ B^T & I_n
\end{pmatrix} \]
We obtain $B=A^T$ and $\Phi = I_m - A^TA$.  

Lastly, the channel associated with \eqref{eq:posdef} is deterministic iff $\Sigma = 0$, i.e. $AA^T = I_n$. Therefore, the category $\S(\gauss)$ is equivalent to the category $\coiso$ of Euclidean co-isometries.

\emph{Independence structure:} Using characterization \eqref{prop:dagger_characterization}, a commutative square in $\S(\gauss)$ is independent if and only if the corresponding square in $\coiso$  
\[\begin{tikzcd}
	{\R^w} && {\R^m} \\
	\\
	{\R^n} && {R^k}
	\arrow["F", from=1-1, to=1-3]
	\arrow["G"', from=1-1, to=3-1]
	\arrow["U", from=1-3, to=3-3]
	\arrow["V"', from=3-1, to=3-3]
\end{tikzcd}\]
satisfies $GF^T = V^TU$. Note that the category of co-isometries $\coiso$ is equivalent to the opposite category of isometries $\iso^\op$ by means of the transposition functor $A \mapsto A^T$. A square of isometries
\[\begin{tikzcd}
	{\R^w} && {\R^m} \\
	\\
	{\R^n} && {R^k}
	\arrow["A"', from=1-3, to=1-1]
	\arrow["B", from=3-1, to=1-1]
	\arrow["J"', from=3-3, to=1-3]
	\arrow["K", from=3-3, to=3-1]
\end{tikzcd}\]
is considered independent if $B^TA = KJ^T$. The category of probability sheaves over $\gauss$ can thus be identified with covariant functors $\iso \to \set$ taking independent squares to pullbacks. Note that by \ref{prop:pushout}, each independent square in $\iso$ will be a pullback, however not every pullback is necessarily independent. For example
\[\begin{tikzcd}
	& {\R^2} \\
	\R && \R \\
	& 0
	\arrow["F", from=2-1, to=1-2]
	\arrow["G"', from=2-3, to=1-2]
	\arrow["0", from=3-2, to=2-1]
	\arrow["0"', from=3-2, to=2-3]
\end{tikzcd} \text{ with } F = \begin{pmatrix} 1 \\ 0 \end{pmatrix}, G = \frac 1 {\sqrt 2} \begin{pmatrix}  1 \\ 1 \end{pmatrix} \]
is a pullback of isometries that is not independent. We thank Matthew Di Meglio for pointing out this counterexample. 

\subsection{Nominal Sets}\label{sec:nom}

\noindent We briefly recall the notion of nominal sets \cite{gabbay1999new,pitts2013nominal}. Let $\A$ be a countably infinite set of names, and $\perm(\A)$ be the group of finite permutations of $\mathbb A$. If $X$ is a set with a $\perm(\A)$-action $(\pi,x) \mapsto \pi \cdot x$, we say that a finite set of names $A \subseteq \A$ \emph{supports} $x \in X$ if
\[ \forall \pi \in \perm(\A), (\forall a \in A, \pi(a) = a) \Rightarrow \pi \cdot x = x \]
We say that $A$ \emph{strongly supports} $x \in X$ if
\[ \forall \pi \in \perm(\A), (\forall a \in A, \pi(a) = a) \Leftrightarrow \pi \cdot x = x \]
A \emph{nominal set} is a $\perm(\A)$-set $X$ in which every element is supported by some finite set of names. A \emph{strong nominal set} is one where each element is strongly supported by a set of names \cite{tzevelekos2009nominal}. Nominal sets and strong nominal sets form categories $\nom$ and $\snom$ respectively, where morphisms $f : X \to Y$ are equivariant functions. 

The \emph{orbit} $\orb(x)$ of an element $x \in X$ is the subset $\{ \pi \cdot x \s \pi \in \perm(\A) \}$. We write $\Pi_0(X)$ for the set of orbits of $X$. A nominal set $X$ is called \emph{atomic} if it consists only of a single orbit. Each nominal set is the coproduct of its orbits
\[ X \cong \bigsqcup_{W \in \Pi_0(X)} W \]

\begin{lemma}\label{lemma:atomic_def}
Let $x \in X$ and $y \in Y$ be elements of nominal sets, where $X$ is strong and atomic. Then there exists at most one equivariant function $f : X \to Y$ with $f(x) = y$, and such a function exists if and only if $\supp(y) \subseteq \supp(x)$.
\end{lemma}
\begin{proof}
By applying \cite[Proposition~15.11]{pitts2013nominal}; the stabilizer condition is discharged by $X$ being strong. 
\end{proof}

\noindent An important role is played by the atomic nominal sets $\rep n$,
\[ \rep n = \{ (a_1,\ldots,a_n) \emph{ distinct names } \} \]
We call a nominal set $X$ \emph{representable} if it is isomorphic to $\rep n$ for some $n \in N$ (this is related to the Yoneda lemma \cite[Exercise~6.1]{pitts2013nominal}).

If $x \in X$ is strongly supported by $A=\{a_1,\ldots,a_n\}$, then the orbit $\orb(x)$ is isomorphic to $\rep n$. This follows immediately from Lemma~\ref{lemma:atomic_def}, by extending the assignment $f(x) = (a_1,\ldots,a_n)$ under equivariance. As a consequence, strong nominal sets are precisely the coproducts of representables. Products and coproducts of strong nominal sets are strong. \\

\noindent\emph{Schanuel topos:} It is well-known that $\nom$ is equivalent to the \emph{Schanuel topos} $\sh(\inj^\op)$ \cite{pitts2013nominal}. Concretely, the Schanuel topos consists of covariant functors $P : \inj \to \set$ preserving pullbacks. Under the equivalence, the nominal sets $\rep n$ correspond to the representable presheaves $\inj(n,-)$. In particular
\begin{equation} \nom(\rep m, \rep n) \cong \inj(n,m) \label{eq:nom_yoneda} \end{equation}
by the Yoneda lemma. \\

\noindent\emph{Name generation:} There is monad $N : \nom \to \nom$ which models fresh name generation (known as the name-generation monad, or free restriction set monad) \cite[Section~9.5]{pitts2013nominal}. The elements of $N(X)$ are equivalence classes of pairs $(A,x)$ where $A$ is a finite set of names and $x \in X$. Two such pairs $(A,x)$ and $(A',x')$ are considered equivalent if $x' = \pi \cdot x$ for some permutation $\pi$ which only interchanges names mentioned in $A,A'$. This models $\alpha$-equivalence; we write the equivalence class of $(A,x)$ as $\name{A}{x}$. For example, we have $\name{a,b}{a} = \name{c}{c}$ by the renaming $\pi = (a\,c)$. 

If $X$ is strong so is $N(X)$, and it holds that $\supp(\name A x) = \supp(x) \setminus A$. The monad restricts to $N : \snom \to \snom$. We denote by $\strongname$ its Kleisli category, which has the structure of a Markov category.

\begin{proposition}
To give a state $I \to X$ in $\strongname$ is to give an orbit $\orb(x)$.
\end{proposition}
\begin{proof}
We have $\nom(I,N(X)) \cong \Pi_0(X)$: To give a point $1\to N(X)$ is to give an element with empty support in $N(X)$, and those are of the form $\langle A \rangle x$ where $A = \supp(x)$. Two such elements are $\alpha$-equivalent iff they lie in the same orbit.
\end{proof}

\noindent For example, the nominal set $\A \times \A$ has two states in $\strongname$: two independent fresh names $\name{a,b}{(a,b)}$, and a single shared fresh name $\name{a}{(a,a)}$.

\begin{proposition}
$\strongname$ has conditionals.
\end{proposition}
\begin{proof}[Proof sketch]
Given $f : A \to N(X \times Y)$, we use Lemma~\ref{lemma:atomic_def} to define the conditional on orbit-by-orbit. If $f(a) = \langle C \rangle (x,y)$, we define
\[ f|_X(a,x) = \langle C' \rangle y \text{ where } C' = C \setminus \{ \supp(x) \} \]
and extend by equivariance. 
\end{proof}
\noindent Note that strongness is a crucial assumption to use Lemma~\ref{lemma:atomic_def}. The Kleisli category of $N$ on $\nom$ does \emph{not} have all conditionals \cite[Proposition~25.21]{stein2021structural}.

\begin{proposition}
Sample spaces in $\strongname$ can be described as pairs $(X,W)$ where $W \subseteq X$ is an orbit. The inclusion $(W,W) \to (X,W)$ is a support inclusion, and the sample space is faithful if and only if $W=X$.
\end{proposition}
\begin{proof}
If $W=\orb(x)$ is an orbit, two equivariant functions $f,g : X \to N(Y)$ are $W$-almost surely equal if $f(x)=g(x)$. The support projection $\pi : X \to N(W)$ is given by 
\[
\pi(x) = \begin{cases}
	\langle \rangle x, &x \in W \\
	\langle \supp(w) \rangle w, &x \notin W
\end{cases}
\]
where $w$ is an arbitrary element of $W$. 
\end{proof}

\begin{proposition}
$\S(\strongname)$ is equivalent to $\inj^\op$
\end{proposition}
\begin{proof}
By Proposition~\ref{prop:faithful}, every sample space in $\S(\strongname)$ is isomorphic to a faithful sample space of the form $(\rep n, \rep n)$. By \eqref{eq:nom_yoneda}, to give a map $\rep m \to \rep n$ is to give an injection $f : n \to m$, and every such map is automatically state-preserving. 
\end{proof}

\end{document}